\documentclass[11pt]{article}
\usepackage[top=30truemm,bottom=30truemm,left=24truemm,right=24truemm]{geometry}

\usepackage{amsmath,amssymb,amsthm}
\usepackage{cases}
\usepackage{url}

\allowdisplaybreaks[4]

\newtheorem{dfn}{Definition}[section]
\newtheorem{lem}[dfn]{Lemma}

\newtheorem{thm}[dfn]{Theorem}

\newtheorem{rem}[dfn]{Remark}

\newcommand{\bff}{{\bold f}}
\newcommand{\bfg}{{\bold g}}

\newcommand{\DV}{{\rm Div}\,}
\newcommand{\dv}{{\rm div}\,}

\newcommand{\BR}{{\Bbb R}}
\newcommand{\BC}{{\Bbb C}}

\newcommand{\tg}{\tilde{g}}

\newcommand{\tf}{\tilde{f}}
\newcommand{\tF}{\tilde{F}}

\newcommand{\tG}{\tilde{G}}

\renewcommand{\th}{\tilde{h}}
\newcommand{\tH}{\tilde{H}}

\newcommand{\tA}{\tilde{A}}
\newcommand{\td}{\tilde{d}}
\newcommand{\tD}{\tilde{D}}

\newcommand{\hu}{\hat{u}}
\newcommand{\hv}{\hat{v}}

\newcommand{\hg}{\hat{g}}
\newcommand{\hh}{\hat{h}}
\newcommand{\heta}{\hat{\eta}}

\newcommand{\htheta}{\hat{\theta}}
\newcommand{\hW}{\hat{W}}

\newcommand{\CF}{{\mathcal F}}

\newcommand{\CL}{{\mathcal L}}
\newcommand{\CM}{{\mathcal M}}

\newcommand{\CR}{{\mathcal R}}

\newcommand{\CH}{{\mathcal H}}

\newcommand{\pd}{\partial}

\newcommand{\uHS}{{\BR^n_+}}
\newcommand{\lHS}{{\BR^n_-}}
\newcommand{\bHS}{{\BR^n_0}}
\newcommand{\dBR}{{\dot{\BR}}}

\newcommand{\eps}{\varepsilon}
\newcommand{\loc}{\mathrm{loc}}

\renewcommand\Re{\operatorname{Re}}

\newcommand{\jump}[1]{\ensuremath{[\![#1]\!]}}
\newcommand{\jjump}[1]{\ensuremath{[\![\![#1]\!]\!]}}

\begin{document}

\title{Solution formula for generalized two-phase Stokes equations and its applications to maximal regularity; model problems}

\author{Naoto Kajiwara \footnote{Applied Physics Course, Department of Electrical, Electronic and Computer Engineering, Gifu University, Yanagido 1-1, Gifu, Gifu 501-1193, JAPAN. E-mail: kajiwara@gifu-u.ac.jp}}

\date{}

\maketitle

%%%%%%%%%%%%%%%%%%%%%%%%%%%%%%%%%%%%%%%%%%%%%%%%%%

\begin{abstract}
In this paper we give a solution formula for the two-phase Stokes equations with and without surface tension and gravity in the whole space with flat interface. 
The solution formula has already considered by Shibata-Shimizu. 
However we reconstruct the formula so that we are able to prove resolvent estimate and maximal regularity estimate. 
In the previous work, they needed to assume additional conditions on normal components. 
We also take care of normal components, while the assumption becomes weaker than before. 
The method is based on an $H^\infty$-calculus which has already used for the Stokes problems with various boundary conditions in the half space. 
\end{abstract}

%%%%%%%%%%%%%%%%%%%%%%%%%%%%%%%%%%%%%%%%%%%%%%%%%%

\begin{flushleft}
\textbf{Keywords} : solution formula, resolvent estimate, maximal regularity, two-phase Stokes equations. 
\end{flushleft}

%%%%%%%%%%%%%%%%%%%%%%%%%%%%%%%%%%%%%%%%%%%%%%%%%%

%\tableofcontents

%%%%%%%%%%%%%%%%%%%%%%%%%%%%%%%%%%%%%%%%%%%%%%%%%%

\section{Introduction}

It is known that the motion of viscous incompressible fluids are governed by the Navier--Stokes equations. 
When we consider two fluids are separated by a free surface, the equations become a challenging free boundary problem. 
Mathematically, the problem is formulated by initial boundary value problems. 
Let $\Omega_+(t)$ and $\Omega_-(t)$ be domains in $\BR^n$ occupied with each fluids, and let they have the same time-independent boundary $\Gamma(t)=\partial\Omega_+(t)(=\partial\Omega_-(t))$ and $\BR^n=\Omega_+(t)\cup\Omega_-(t)\cup\Gamma(t)$. 
The unknowns are the boundary $\Gamma(t)$ as well as the velocity $v(x,t)={}^t(v_1, \ldots, v_n)$ and pressure $\theta(x,t)$ defined on $\Omega(t)=\Omega_+(t)\cup\Omega_-(t)$. 
The equations are known as follows; 
\begin{equation}
\left\{\begin{aligned}
\rho(\pd_t v + (v\cdot \nabla)v) -  \DV S(v,\theta)=0 &\qquad{\rm in}~\Omega(t),~t>0, \\
\dv u = 0 &\qquad {\rm in}~\Omega(t),~t>0, \\
\jump{S(v,\theta)\nu_t}=c_\sigma\CH \nu_t + \jump{\rho}c_g x_n \nu_t &\qquad {\rm on}~\Gamma(t),~t>0, \\
\jump{v}=0 &\qquad {\rm on}~\Gamma(t),~t>0, \\
V=v\cdot \nu_t&\qquad {\rm on}~\Gamma(t),~t>0, \\
v|_{t=0}=v_0&\qquad {\rm in}~\Omega(0). \\
\end{aligned}\right.
\end{equation}
Here, $S(v,\theta)=\mu D(v)-\theta I=(\mu(\pd_i v_j + \pd_j v_i) - \delta_{ij}\theta)_{ij}$ is $n$ times $n$ symmetric stress tensor, $V$ is the normal velocity of $\Gamma(t)$, $\nu_t$ is the unit outward normal vector pointing from $\Omega_+(t)$ to $\Omega_-(t)$ and $\CH$ is the mean curvature of $\Gamma(t)$. 
The letters $\rho, \mu, c_\sigma$ and $c_g$ denote the coefficients of density, viscosity, surface tension and gravity, respectively, where $\rho$ and $\mu$ are constants on each domain $\Omega_{\pm}(t)$.  
The symbol $\jump{\cdot}$ denotes jump across the interface $\Gamma(t)$. 
For example, for piecewise constant density $\rho$ defined on $\Omega(t)$, the quantity $\jump{\rho}$ means $\jump{\rho}=\rho|_{\Omega_+(t)} - \rho|_{\Omega_-(t)}$. 

It is known that Hanzawa transform is a useful method to solve free boundary problems. 
In this method, the unknown $\Gamma(t)$ is given by a height function defined on the boundary of a fixed domain. 
After using this transformation, the equations become quasi-linear equations. 
Therefore it is important to consider their linearized equations. 
In addition to above discussion, maximal regularity for the linearized equations on the whole space with flat interface is a key; 
\begin{equation}\label{Swith}
\left\{\begin{aligned}
\rho\pd_t U -  \DV S(U,\Theta)=F &\qquad{\rm in}~\dBR^n(:=\uHS\cup\lHS), t>0,   \\
\dv U = F_d &\qquad {\rm in}~\dBR^n, t>0, \\
\pd_t Y +  U_n = D &\qquad{\rm on}~\BR^n_0(:=\pd\uHS), t>0, \\
\jump{S(U,\Theta)\nu}-(\jump{\rho}c_g + c_\sigma \Delta')Y\nu=\jump{G} &\qquad{\rm on}~\bHS, t>0, \\
\jump{U}=\jump{H}&\qquad{\rm on}~\BR^n_0, t>0, \\
(U,Y)|_{t=0} = (0,0)&\qquad{\rm in}~\dBR^n
\end{aligned}\right.
\end{equation}
where $F, F_d, D, G$ and $H$ are external forces, and $\nu=(0,\ldots, 0,-1)$. 
Moreover, we consider the corresponding resolvent equations and the case $c_\sigma=c_g=0$; 
\begin{equation}\label{RSwith}
\left\{\begin{aligned}
\rho\lambda u -  \DV S(u,\theta)=F &\qquad{\rm in}~\dBR^n,   \\
\dv u = f_d &\qquad {\rm in}~\dBR^n, \\
\lambda \eta + u_n = d &\qquad{\rm on}~\BR^n_0, \\
\jump{S(u,\theta)\nu}-(\jump{\rho}c_g + c_\sigma \Delta')\eta\nu=\jump{g} &\qquad{\rm on}~\bHS, \\
\jump{u}=\jump{h}&\qquad{\rm on}~\BR^n_0, 
\end{aligned}\right.
\end{equation}
\begin{equation}\label{Swithout}
\left\{\begin{aligned}
\rho\pd_t U -  \DV S(U,\Theta)=F &\qquad{\rm in}~\dBR^n, t>0,   \\
\dv U = F_d &\qquad {\rm in}~\dBR^n, t>0, \\
\jump{S(U,\Theta)\nu}=\jump{G} &\qquad{\rm on}~\bHS, t>0, \\
\jump{U}=\jump{H}&\qquad{\rm on}~\BR^n_0, t>0, \\
U|_{t=0} = 0&\qquad{\rm in}~\BR^{n-1}, 
\end{aligned}\right.
\end{equation}
\begin{equation}\label{RSwithout}
\left\{\begin{aligned}
\rho\lambda u -  \DV S(u,\theta)= f &\qquad{\rm in}~\dBR^n,   \\
\dv u = f_d &\qquad {\rm in}~\dBR^n, \\
\jump{S(u,\theta)\nu}=\jump{g} &\qquad{\rm on}~\bHS, \\
\jump{u}=\jump{h}&\qquad{\rm on}~\BR^n_0. 
\end{aligned}\right.
\end{equation}
In this paper we construct the solution formulas of these four problems. 
The approach is based on the standard way, which means partial Fourier transforms and Laplace transforms of the equations. 
When we solve ordinary differential equations, we need to take care of a matrix. 
In the previous work \cite{SS11}, they also gave the solution formulas by analyzing the ordinary differential equations and the matrix.
However, our approach will be easier than before. 
We focus only on the determinant of the matrix and the order of growth of the cofactor matrix. 
Then we are able to get the solution formulas clearly. 
This is one of our main theorems. 
As its application, we are able to prove resolvent estimate and maximal regularity estimate. 
When we can get solution formulas with a suitable form, we know that they have these estimates. 
This strategy has shown in the paper \cite{K22}, which considered the Stokes equations with various boundary conditions in the half space. 
We remark that the paper \cite{SS11} had to assume additional conditions for $h_n$ and $H_n$. 
On the other hand, we can relax some conditions. 
The quantity of the calculation is much less than before too. 

There are several papers on the two-phase free boundary problems. 
The problems can be divided into two cases; one is a compact free surface and the other is a non-compact one. 
We review only for the first case. 
Tanaka \cite{T93} proved the global existence theorem in $L^2$ Sobolev-Slobodetskii space. 
Denisova proved the same results with $c_\sigma=0$ in both H\"older space \cite{D07} and $L^2$-based Sobolev space \cite{D14}. 
Denisova et al.~extended their results for capillarity fluids, i.e. $c_\sigma>0$ in both whole space \cite{D94} and bounded domain \cite{DS11}. 
Shimizu \cite{S08} treated the case that $c_\sigma=0$ in $L^p$-$L^q$ settings. 
K\"ohne, Pr\"uss and Wilke proved global well-posedness for the capillarity fluids in $L^p$-settings and their asymptotic behaviour in \cite{KPW13}. 
Saito and Shibata \cite{SS20} considered comprehensive approach for two-phase problems. 
Moreover there are some papers for two-phase problems e.g. varifold solution \cite{A07} and viscosity solution \cite{GT94, T95}. 
Concerning resolvent estimates and maximal regularity, see also \cite{K22, KS12, PS16, SS03, SS08, SS09, SS11, SS12}. 

This paper is organized as follows. 
First we introduce some notations and state our main theorems in section \ref{main}. 
The main point of this paper is shorten the proof of estimates, and weaken the assumption on the normal components compared with the previous work \cite{SS11}. 
In section \ref{reduction}, we cite some theorems from the paper \cite{SS11}, which is one of the standard way to consider solution formulas. 
This implies that it is enough to consider the cases $f=f_d=0$ and $F=F_d=0$. 
The solution formula from the boundary data is the most important part. 
This is done in section \ref{formula} for the equations \eqref{Swithout} and \eqref{RSwithout}. 
And then, in section \ref{proof}, we prove resolvent $L_p$ estimate and maximal $L_p$-$L_q$ estimate, which is based on the theorem in \cite{K22}. 
Analysis of the equations \eqref{Swith} and \eqref{RSwith} is given in section \ref{final}. 
The solution formulas and the estimates depend on the results on the equations \eqref{Swithout} and \eqref{RSwithout}

\section{Main theorem}\label{main}
In this section we prepare some notations and function spaces and we give the main theorems. 
Let $\uHS$, $\lHS$ and $\bHS$ be the upper and lower half-space and its flat boundary and let $Q_+$, $Q_-$ and $Q_0^n$ be the corresponding time-space domains; 
\begin{align*}
&\uHS:=\{x=(x_1, \ldots, x_n)\in\BR^n\mid x_n>0\}, \quad \lHS:=\{x=(x_1, \ldots, x_n)\in\BR^n\mid x_n<0\}, \\
&\bHS:=\{x=(x',0)=(x_1, \ldots, x_{n-1},0)\in\BR^n\},\\
&Q_+:=\uHS\times (0,\infty), \quad Q_-:=\lHS\times (0,\infty), \quad Q_0:=\bHS\times (0,\infty). 
\end{align*}

Given a domain $D$, Lebesgue and Sobolev spaces are denoted by $L_q(D)$ and $W^m_q(D)$ with norms $\|\cdot\|_{L_q(D)}$ and $\|\cdot\|_{W^m_q(D)}$. 
Same manner is applied in the $X$-valued spaces $L_p(\BR, X)$ and $W^m_p(\BR, X)$. 
For a scalar function $f$ and $n$-vector $\bff=(f_1, \ldots, f_n)$, we use the following symbols; 
\begin{alignat*}{3}
&\nabla f = (\pd_1 f, \ldots, \pd_n f), &\quad &\nabla^2 f = (\pd_i\pd_j f \mid i,j=1,\ldots, n), \\
&\nabla \bff = (\pd_i f_j \mid i,j=1, \ldots, n), &\quad &\nabla^2 \bff = (\pd_i \pd_j f_k \mid i,j,k= 1, \ldots, n). 
\end{alignat*}
Even though $\bfg=(g_1, \ldots, g_{\tilde{n}}) \in X^{\tilde{n}}$ for some $\tilde{n}$, we use the notations $\bfg\in X$ and $\|\bfg\|_X$ as $\sum_{j=1}^{\tilde{n}}\|g_j\|_X$ for simplicity. 
Set 
\begin{align*}
\hW^1_q(D)= \{\pi\in L_{q,\loc}(D)\mid\nabla \pi\in L_q(D)\}, \quad \hW^1_{q,0}(D)=\{\pi\in \hW^1_q(D)\mid \pi|_{\pd D}=0\}
\end{align*}
and let $\hW^{-1}_q(D)$  denote the dual space of $\hW^1_{q',0}(D)$, where $1/q+1/q'=1$. 
For $\pi\in \hW^{-1}_q (D)\cap L_q(D)$, we have 
\[\|\pi\|_{\hW^{-1}_q(D)} = \sup\left\{ \left|\int_D \pi \phi dx\right| \mid \phi \in \hW^1_{q',0}(D), \|\nabla \phi\|_{L_{q'}(D)}=1\right\}. \]
Although we consider time interval $\BR_+$, we regard functions on $\BR$ to use Fourier transform. 
To do so and to consider Laplace transforms as Fourier transforms, we introduce some function spaces; 
\begin{align*}
L_{p,0,\gamma_0}(\BR, X)&:=\{f:\BR\to X \mid e^{-\gamma_0 t}f(t)\in L_p(\BR, X),~f(t)=0~\text{for}~t<0\}, \\
W^m_{p,0,\gamma_0}(\BR, X)&:=\{f\in L_{p,0,\gamma_0}(\BR, X) \mid e^{-\gamma_0 t}\pd_t^j f(t) \in L_p(\BR, X),~j=1,\ldots, m\}, \\
L_{p,0}(\BR, X)&:=L_{p,0,0}(\BR; X),\quad W^m_{p,0}(\BR, X):=W^m_{p,0,0}(\BR; X)
\end{align*}
for some $\gamma_0\ge0$. 
Let $\CF$ and $\CF^{-1}$ denote the Fourier transform and its inverse, defined as 
\begin{align*}
\CF[f](\xi) =\CF_x [f](\xi)   = \int_{\BR^n} e^{-i x\cdot \xi} f(x) dx, \quad \CF^{-1}[g](x)= \CF^{-1}_\xi[g](x)  = \frac{1}{(2\pi)^n}\int_{\BR^n} e^{ix\cdot \xi} g(\xi) d\xi. 
\end{align*}
Similarly, let $\CL$ and $\CL^{-1}_\lambda$ denote two-sided Laplace transform and its inverse, defined as 
\begin{align*}
\CL[f](\lambda) = \int_{-\infty}^\infty e^{-\lambda t} f(t) dt, \quad \CL^{-1}_\lambda[g](t) = \frac{1}{2\pi}\int_{-\infty}^\infty e^{\lambda t} g(\lambda) d\tau, 
\end{align*}
where $\lambda = \gamma + i\tau \in \BC$. 
Given $s\ge 0$ and $X$-valued function $f$, we use the following Bessel potential spaces to treat fractional orders; 
\begin{align*}
H^s_{p,0,\gamma_0}(\BR, X) &:= \{f:\BR\to X \mid \Lambda^s_\gamma f := \CL^{-1}_\lambda[|\lambda|^s \CL[f](\lambda)](t) \in L_{p, 0, \gamma}(\BR, X)~\text{for~any}~\gamma\ge \gamma_0\}, \\
H^s_{p,0}(\BR, X)&:=H^s_{p,0,0}(\BR, X). 
\end{align*}
Since we need to take care of the $n$-th component of the velocity, we introduce the following function spaces; 
\begin{align*}
E_q(\dBR^n)&:=\{h_n \in W^2_q(\dBR^n)\mid |\nabla'|^{-1} \pd_n h_n := \CF_{\xi'}^{-1} |\xi'|^{-1} \CF_{x'} (\pd_n h_n)(x', x_n) \in L_q(\dBR^n)\}. 
\end{align*}
We remark that this condition is weaker than the paper \cite{SS11} since they assumed $|\nabla'|^{-1} h_n \in \hW^1_q(\dBR^n)$. 
It will be easier to use this assumption to handle a difficult term.   
Let $\Sigma_{\eps, \gamma}:=\{\lambda\in\BC\setminus\{0\}\mid |\arg\lambda|<\pi-\eps, |\lambda|\ge\gamma\}$ and $\Sigma_{\eps}:=\Sigma_{\eps, 0}$. 
Throughout this paper, let $\rho, \mu$ be positive constants on each domain $\BR^n_{\pm}$, denoted by $\rho_{\pm}$ and $\mu_{\pm}$. 
We are ready to state our main results. 

\begin{thm}\label{resolventthm}
Let $0<\eps<\pi/2$ and $1<q<\infty$. 
Then for any $\lambda \in \Sigma_\eps$, 
\[ f\in L_q(\dBR^n),\quad f_d\in \hW^{-1}_q(\BR^n)\cap W^1_q(\dBR^n), \quad g\in W^1_q(\dBR^n), \quad h \in W^2_q(\dBR^n), \quad h_n\in E_q(\dBR^n)\]
problem \eqref{RSwithout} admits a unique solution $(u, \theta) \in W^2_q(\dBR^n) \times \hW^1_q(\dBR^n)$ with the resolvent estimate; 
\begin{align*}
&\|(|\lambda|u, |\lambda|^{1/2}\nabla u, \nabla^2 u, \nabla \theta)\|_{L_q(\dBR^n)} \\
\le 
&C_{n, q, \eps} \left\{\|(f, |\lambda|^{1/2} f_d, \nabla f_d, |\lambda|^{1/2} g, \nabla g, |\lambda|h, \nabla^2 h, |\lambda||\nabla'|^{-1}\pd_n h_n) \|_{L_q(\dBR^n)} + |\lambda| \|f_d\|_{\hW^{-1}_q(\dBR^n)}\right\}. 
\end{align*} 
\end{thm}

\begin{thm}\label{maxregthm}
Let $1<p, q<\infty$ and $\gamma_0\ge0$. 
Then for any 
\begin{align*}
&F\in L_{p,0,\gamma_0}(\BR, L_q(\dBR^n)),\quad F_d \in W^1_{p,0,\gamma_0}(\BR, \hW^{-1}_q(\BR^n))\cap L_{p,0,\gamma_0}(\BR, W^1_q(\dBR^n)), \\
&G\in H^{1/2}_{p,0,\gamma_0}(\BR, L_q(\dBR^n))\cap L_{p,0,\gamma_0}(\BR, W^1_q(\dBR^n)),  \\
&H\in W^1_{p,0,\gamma_0}(\BR, L_q(\dBR^n)) \cap L_{p,0,\gamma_0}(\BR, W^2_q(\dBR^n)), \quad 
H_n\in W^1_{p,0,\gamma_0}(\BR, E_q(\dBR^n)), 
\end{align*}
problem \eqref{Swithout} admits a unique solution $(U, \Theta)$ such that 
\begin{align*}
U&\in W^1_{p,0,\gamma_0}(\BR, L_q(\dBR^n)) \cap L_{p,0,\gamma_0}(\BR, W^2_q(\dBR^n)), \\
\Theta &\in L_{p,0,\gamma_0}(\BR, \hW^1_q(\dBR^n))
\end{align*} 
with the maximal $L_p$-$L_q$ regularity estimate; 
\begin{align*}
&\|e^{-\gamma t}(\pd_t U, \gamma U, \Lambda^{1/2}_\gamma \nabla U, \nabla^2 U, \nabla \Theta)\|_{L_p(\BR, L_q(\dBR^n))} \\
\le 
&C_{n, p, q, \gamma_0} \left\{\|e^{-\gamma t} (F, \Lambda^{1/2}_\gamma F_d, \nabla F_d, \Lambda^{1/2}_\gamma G, \nabla G, \pd_t H, \nabla^2 H, \pd_t(|\nabla'|^{-1}\pd_n H_n) \|_{L_p(\BR, L_q(\dBR^n))} \right.\\
&\qquad \qquad \left.+ \|e^{-\gamma t} (\pd_t F_d, \gamma F_d)\|_{L_p(\BR, \hW^{-1}_q(\BR^n))} \right\}, 
\end{align*}
for any $\gamma \ge \gamma_0$. 
\end{thm}

We can extend above theorems to the problem \eqref{Swith} and \eqref{RSwith}. 
Let $c_\sigma>0$ and $c_g>0$. 

\begin{thm}\label{resolventthm2}
Let $0<\eps<\pi/2$ and $1<q<\infty$. 
Then there exists a constant $\gamma_0\ge1$ depending on $\eps>0$ such that for any $\lambda \in \Sigma_{\eps, \gamma_0}$, 
\begin{align*}
&f\in L_q(\dBR^n),\quad f_d\in \hW^{-1}_q(\BR^n)\cap W^1_q(\dBR^n), \quad g\in W^1_q(\dBR^n), \\
&h \in W^2_q(\dBR^n), \quad h_n\in E_q(\dBR^n), \quad d\in W^2_q(\dBR^n)
\end{align*}
problem \eqref{RSwith} admits a unique solution $(u, \theta, \eta) \in W^2_q(\dBR^n) \times \hW^1_q(\dBR^n)\times W^3_q(\dBR^n)$ with the resolvent estimate; 
\begin{align*}
&\|(|\lambda|u, |\lambda|^{1/2}\nabla u, \nabla^2 u, \nabla \theta)\|_{L_q(\dBR^n)} + |\lambda|\|\eta\|_{W^2_q(\dBR^n)} + \|\eta\|_{W^3_q(\dBR^n)} \\
\le 
&C_{n, q, \eps, \gamma_0}\left\{\|(f, |\lambda|^{1/2} f_d, \nabla f_d, |\lambda|^{1/2} g, \nabla g, |\lambda|h, \nabla^2 h, |\lambda||\nabla'|^{-1}\pd_n h_n) \|_{L_q(\dBR^n)} + |\lambda| \|f_d\|_{\hW^{-1}_q(\dBR^n)} + \|d\|_{W^2_q(\dBR^n)}\right\}. 
\end{align*} 
Moreover we have 
\begin{align*}
|\lambda|^{3/2}\|\eta\|_{W^1_q(\dBR^n)} \le &C_{n, q, \eps,\gamma_0}\left\{\|(f, |\lambda|^{1/2} f_d, \nabla f_d, |\lambda|^{1/2} g, \nabla g, |\lambda|h, \nabla^2 h, |\lambda||\nabla'|^{-1}\pd_n h_n) \|_{L_q(\dBR^n)} \right.\\
&\left. \qquad \qquad \qquad + |\lambda| \|g\|_{\hW^{-1}_q(\dBR^n)} + \|d\|_{W^2_q(\dBR^n)} + |\lambda|^{1/2}\|d\|_{W^1_q(\dBR^n)}\right\}\end{align*}
and 
\begin{align*}
|\lambda|^{2}\|\eta\|_{L_q(\dBR^n)} \le &C_{n, q, \eps,\gamma_0}\left\{\|(f, |\lambda|^{1/2} f_d, \nabla f_d, |\lambda|^{1/2} g, \nabla g, |\lambda|h, \nabla^2 h, |\lambda||\nabla'|^{-1}\pd_n h_n) \|_{L_q(\dBR^n)} \right.\\
&\left. \qquad \qquad \qquad + |\lambda| \|g\|_{\hW^{-1}_q(\dBR^n)} + \|d\|_{W^2_q(\dBR^n)} + |\lambda|\|d\|_{L_q(\dBR^n)}\right\}. 
\end{align*}
\end{thm}

\begin{thm}\label{maxregthm2}
Let $1<p, q<\infty$. 
Then there exists a constant $\gamma_0\ge1$ such that for any 
\begin{align*}
F\in L_{p,0,\gamma_0}(\BR, L_q(\dBR^n))&,\quad F_d \in W^1_{p,0,\gamma_0}(\BR, \hW^{-1}_q(\BR^n))\cap L_{p,0,\gamma_0}(\BR, W^1_q(\dBR^n)), \\
G\in H^{1/2}_{p,0,\gamma_0}(\BR, L_q(\dBR^n))\cap L_{p,0,\gamma_0}(\BR, W^1_q(\dBR^n))&,\quad H\in W^1_{p,0,\gamma_0}(\BR, L_q(\dBR^n)) \cap L_{p,0,\gamma_0}(\BR, W^2_q(\dBR^n)), \\
H_n\in W^1_{p,0,\gamma_0}(\BR, E_q(\dBR^n))&, \quad D\in L_{p,0,\gamma_0}(\BR, W^2_q(\dBR^n)), 
\end{align*}
problem \eqref{Swith} admits a unique solution $(U, \Theta, Y)$ such that 
\begin{align*}
U&\in W^1_{p,0,\gamma_0}(\BR, L_q(\dBR^n)) \cap L_{p,0,\gamma_0}(\BR, W^2_q(\dBR^n)), \\
\Theta &\in L_{p,0,\gamma_0}(\BR, \hW^1_q(\dBR^n)), \\
Y & \in L_{p,0,\gamma_0}(\BR, W^3_q(\dBR^n))\cap W^1_{p,0,\gamma_0}(\BR, W^2_q(\dBR^n))
\end{align*} 
with the maximal $L_p$-$L_q$ regularity estimate; 
\begin{align*}
&\|e^{-\gamma t}(\pd_t U, \gamma U, \Lambda^{1/2}_\gamma \nabla U, \nabla^2 U, \nabla \Theta)\| _{L_p(\BR, L_q(\dBR^n))} + \|e^{-\gamma t}(\pd_t Y, \gamma Y)\|_{L_p(\BR, W^2_q(\dBR^n))} + \|e^{-\gamma t} Y\|_{L_p(\BR, W^3_q(\dBR^n))}\\
\le 
&C_{n, p, q, \gamma_0}\left\{\|e^{-\gamma t} (F, \Lambda^{1/2}_\gamma F_d, \nabla F_d, \Lambda^{1/2}_\gamma G, \nabla G, \pd_t H, \nabla^2 H, \pd_t (|\nabla'|^{-1}\pd_n H_n)) \|_{L_p(\BR, L_q(\dBR^n))} \right.\\
&\qquad \qquad \left.+ \|e^{-\gamma t} (\pd_t F_d, \gamma F_d)\|_{L_p(\BR, \hW^{-1}_q(\BR^n))} + \|e^{-\gamma t} D\|_{L_p(\BR, W^2_q(\dBR^n))}\right\}, 
\end{align*}
for any $\gamma \ge \gamma_0$. 
Moreover we have that if $D\in H^{1/2}_{p,0, \gamma_0}(\BR, W^1_q(\dBR^n))$ in addition, then $Y\in H^{3/2}_{p,0, \gamma_0}(\BR, W^1_q(\dBR^n))$ and 
\begin{align*}
&\|e^{-\gamma t} \Lambda^{3/2}_\gamma Y\|_{L_p(\BR, W^1_q(\dBR^n))} \\
\le &C_{n, p, q, \gamma_0} \left\{\|e^{-\gamma t} (F, \Lambda^{1/2}_\gamma F_d, \nabla F_d, \Lambda^{1/2}_\gamma G, \nabla G, \pd_t H, \nabla^2 H, \pd_t (|\nabla'|^{-1}\pd_n H_n)) \|_{L_p(\BR, L_q(\dBR^n))} \right.\\
&\qquad \left.+ \|e^{-\gamma t} (\pd_t F_d, \gamma F_d)\|_{L_p(\BR, \hW^{-1}_q(\BR^n))} + \|e^{-\gamma t} D\|_{L_p(\BR, W^2_q(\dBR^n))} + \|e^{-\gamma t}\Lambda^{1/2}_\gamma D\|_{L_p(\BR, W^1_q(\dBR^n))}\right\}
\end{align*}
for any $\gamma\ge \gamma_0$. 
Moreover we have that if $D\in W^1_{p,0, \gamma_0}(\BR, L_q(\dBR^n))$ in addition, then $Y\in W^2_{p,0, \gamma_0}(\BR, L_q(\dBR^n))$ and 
\begin{align*}
&\|e^{-\gamma t} \Lambda^2_\gamma Y\|_{L_p(\BR, L_q(\dBR^n))} \\
\le &C_{n, p, q, \gamma_0} \left\{\|e^{-\gamma t} (F, \Lambda^{1/2}_\gamma F_d, \nabla F_d, \Lambda^{1/2}_\gamma G, \nabla G, \pd_t H, \nabla^2 H, \pd_t (|\nabla'|^{-1}\pd_n H_n)) \|_{L_p(\BR, L_q(\dBR^n))} \right.\\
&\qquad \left.+ \|e^{-\gamma t} (\pd_t F_d, \gamma F_d)\|_{L_p(\BR, \hW^{-1}_q(\BR^n))} + \|e^{-\gamma t} D\|_{L_p(\BR, W^2_q(\dBR^n))} + \|e^{-\gamma t}\Lambda_\gamma D\|_{L_p(\BR, L_q(\dBR^n))}\right\}
\end{align*}
for any $\gamma\ge \gamma_0$. 
\end{thm}

\begin{rem}
By interpolation theory, we have 
\begin{align*}
W^1_{p,0, \gamma_0}(\BR, L_q(\dBR^n)) \cap L_{p,0,\gamma_0}(\BR, W^2_q(\dBR^n)) &\subset H^{1/2}_{p,0, \gamma_0}(\BR, W^1_q(\dBR^n)), \\
W^2_{p,0, \gamma_0}(\BR, L_q(\dBR^n))\cap W^1_{p,0,\gamma_0}(\BR, W^2_q(\dBR^n))&\subset H^{3/2}_{p,0, \gamma_0}(\BR, W^1_q(\dBR^n)). 
\end{align*}
\end{rem}

\section{Reduction to the problem only with boundary data}\label{reduction}
In this section we follow the paper \cite{SS11} so that it is enough to consider the case $f=f_d=0$ and $F=F_d=0$ by subtracting solutions of inhomogeneous data. 

We start with whole space problems. 
\begin{lem}[{\cite[Lemma 2.1]{SS11}}]
Let $1<p, q<\infty$ and $\gamma_0\ge0$. \\
{\rm (1)} For any $f_d\in \hW^{-1}_q(\BR^n)\cap W^1_q(\dBR^n)$, there exists a $z\in W^2_q(\dBR^n)$ such that $\dv z = f_d$ in $\dBR^n$, $\jump{z}=0$ on $\dBR^n_0$ and there hold the estimates: 
\begin{align*}
\|z\|_{L_q(\dBR^n)}&\le C_{n,q}\|f_d\|_{\hW^{-1}_q(\dBR^n)}, \\
\|\nabla^{j+1} z\|_{L_q(\dBR^n)}&\le C_{n,q}\|\nabla^j f_d\|_{L_q(\dBR^n)}~(j=0,1). 
\end{align*}
{\rm (2)} For any $F_d\in W^1_{p,0,\gamma_0}(\BR, \hW^{-1}_q(\BR^n)) \cap L_{p,0,\gamma_0}(\BR, W^1_q(\dBR^n))$, there exists a 
\[Z\in W^1_{p,0,\gamma_0}(\BR, L_q(\dBR^n))\cap L_{p,0,\gamma_0}(\BR, W^2_q(\dBR^n))\] 
such that $\dv Z = F_d$ in $\dBR^n\times \BR$, $\jump{Z(t)}=0$ on $\dBR^n_0\times \BR$ and there hold the estimates: 
\begin{align*}
\|e^{-\gamma t} (\pd_t Z, \gamma Z)\|_{L_p(\BR, L_q(\dBR^n))} &\le C_{n,p,q}\|e^{-\gamma t} (\pd_t F_d, \gamma F_d)\|_{L_p(\BR, \hW^{-1}_q(\BR^n))}, \\
\|e^{-\gamma t}\Lambda^{1/2}_\gamma \nabla Z\|_{L_p(\BR, L_q(\dBR^n))}&\le C_{n,p,q}\|e^{-\gamma t}\Lambda^{1/2}_\gamma F_d\|_{L_p(\BR, L_q(\dBR^n))}, \\
\|e^{-\gamma t}\nabla^2 Z\|_{L_p(\BR, L_q(\dBR^n))} &\le C_{n,p,q}\|e^{-\gamma t}\nabla F_d\|_{L_p(\BR, L_q(\dBR^n))}
\end{align*}
for any $\gamma\ge\gamma_0$. 
\end{lem}

Setting $u=v+z, \tf=f-(\rho\lambda z-\Delta z)$ and $U=V+Z, \tF=F-(\rho \pd_t Z-\Delta Z)$, we would like to find $(v,\theta), (V, \Theta)$ such that 
\begin{equation}\label{RS f_d=0}
\left\{\begin{aligned}
\rho\lambda v -  \DV S(v,\theta)= \tf &\qquad{\rm in}~\dBR^n,   \\
\dv v = 0 &\qquad {\rm in}~\dBR^n, \\
\jump{S(v,\theta)\nu}=\jump{g - \mu D(z)\nu}&\qquad{\rm on}~\bHS, \\
\jump{v}=\jump{h}&\qquad{\rm on}~\BR^n_0. 
\end{aligned}\right.
\end{equation}
and 
\begin{equation}\label{S F_d=0}
\left\{\begin{aligned}
\rho\pd_t V -  \DV S(V,\Theta)=\tF &\qquad{\rm in}~\dBR^n, t>0,   \\
\dv U = 0 &\qquad {\rm in}~\dBR^n, t>0, \\
\jump{S(V,\Theta)\nu}=\jump{G - \mu D(Z)\nu} &\qquad{\rm on}~\bHS, t>0, \\
\jump{V}=\jump{H}&\qquad{\rm on}~\BR^n_0, t>0, \\
V|_{t=0} = 0&\qquad{\rm in}~\BR^{n-1}, 
\end{aligned}\right.
\end{equation}
We have that 
\begin{align*}
\|\tf\|_{L_q(\dBR^n)}&\le \|f\|_{L_q(\dBR^n)} + C_{n,q}(|\lambda|\|f_d\|_{\hW^{-1}_q(\BR^n)} + \|\nabla f_d\|_{L_q(\dBR^n)}), \\
\|e^{-\gamma t}\tF\|_{L_p(\BR, L_q(\dBR^n))}&\le \|e^{-\gamma t} F\|_{L_p(\BR, L_q(\dBR^n))} + C_{n,p,q}(\|e^{-\gamma t}\pd_t F_d\|_{L_p(\BR, \hW^{-1}_q(\BR^n))} + \|e^{-\gamma t} \nabla F_d\|_{L_p(\BR, L_q(\dBR^n))}).  
\end{align*}
Therefore we can reduce the problem $f_d=0, F_d=0$. 

Second, we would like to reduce the case $f=0$, $F=0$. 
Let $P(\xi)=(P_{j,k})_{jk}=(\delta_{jk}-\xi_j\xi_k|\xi|^{-2})_{jk}$ be the Helmholtz decomposition. 
Then functions 
\begin{alignat*}{3}
\psi_\pm (x)=\CF^{-1}_\xi \left[\frac{P(\xi) \CF_x f(\xi)}{\rho_\pm \lambda+\mu_\pm |\xi|^2}\right](x)&, \quad 
\phi_\pm (x)= -\CF^{-1}_\xi\left[\frac{i\xi\cdot\CF_x f(\xi)}{|\xi|^2}\right](x), \\
\Psi_\pm (x, t)= \CL_\lambda \CF^{-1}_\xi \left[\frac{P(\xi) \CF_x \CL F(\xi, \lambda)}{\rho_\pm\lambda+\mu_\pm|\xi|^2}\right](x, t)&, \quad 
\Phi_\pm(x, t)= -\CL_\lambda\CF^{-1}_\xi\left[\frac{i\xi\cdot\CF_x\CL F(\xi, \lambda)}{|\xi|^2}\right](x, t)
\end{alignat*}
satisfy 
\begin{align*}
&(\psi_\pm, \phi_\pm)\in W^2_q(\BR^n)\times \hW^1_q(\BR^n), \\
&\rho_\pm\lambda \psi_\pm - \mu_\pm \Delta \psi_\pm + \nabla \phi_\pm = f,\quad \dv \psi_\pm=0 \quad {\rm in}~\BR^n, \\
&\|(|\lambda|\psi_\pm, |\lambda|^{1/2}\nabla \psi_\pm, \nabla^2 \psi_\pm, \nabla \phi_\pm)\|_{L_q(\BR^n_\pm)} \le C_{n,q,\eps}\|f\|_{L_q(\dBR^n)}
\end{align*}
and 
\begin{align*}
&\Psi_\pm\in W^1_{p,0, \gamma_0}(\BR, L_q(\BR^n))\cap L_{p,0, \gamma_0}(\BR, W^2_q(\BR^n)),\quad \Phi_\pm\in L_{p,0, \gamma_0}(\BR, \hW^1_q(\BR^n)), \\
&\rho_\pm\pd_t \Psi_{\pm} - \mu_\pm \Delta \Psi_\pm + \nabla \Phi_\pm= F,\quad \dv \Psi_\pm=0 \quad {\rm in}~\BR^n\times (0,\infty),\quad \Psi_\pm |_{t=0}=0, \\
&\|e^{-\gamma t}(\pd_t \Psi_\pm, \gamma \Psi_\pm, \Lambda_\gamma^{1/2} \nabla \Psi, \nabla^2 \Psi_\pm, \nabla \Phi_\pm)\|_{L_p(\BR, L_q(\BR^n_\pm))} \le C_{n,p,q, \gamma_0}\|e^{-\gamma t} F\|_{L_p(\BR, L_q(\dBR^n))}, 
\end{align*}
for any $1<p,q<\infty$, $\gamma\ge\gamma_0\ge0, f\in L_q(\dBR^n), F\in L_{p,0,\gamma_0}(\BR, L_q(\dBR^n))$ and $\lambda\in\Sigma_\eps$ with $0<\eps<\pi/2$. 
We define 
\begin{align*}
(\psi, \phi, \Psi, \Phi):=\begin{cases} (\psi_+, \phi_+, \Psi_+, \Phi_+)~{\rm for}~x\in\BR^n_+ ,\\ (\psi_-, \phi_-, \Psi_-, \Phi_-)~{\rm for}~x\in\BR^n_-, \end{cases}
\end{align*}
then we have $\jump{\phi}=0$ on $\BR^n_0$ and $\jump{\Phi(t)}=0$ on $Q_0$. 

Setting $u:=\psi+w$, $\theta:=\phi+\kappa$ in \eqref{RSwithout} with $f_d=0$, and $U:=\Psi+W$, $\Theta=\Phi+\Xi$ in \eqref{Swithout} with $F_d=0$, respectively, we have  
\begin{equation}\label{RSwithout f=f_d=0}
\left\{\begin{aligned}
\rho\lambda w -  \mu \Delta w + \nabla \kappa = 0 &\qquad{\rm in}~\dBR^n,   \\
\dv w = 0 &\qquad {\rm in}~\dBR^n, \\
\jump{S(w,\kappa)\nu}=\jump{g-\mu D(\psi)\nu} &\qquad{\rm on}~\bHS, \\
\jump{w}=\jump{h-\psi}&\qquad{\rm on}~\BR^n_0. 
\end{aligned}\right.
\end{equation}
and 
\begin{equation}\label{Swith f=f_d=0}
\left\{\begin{aligned}
\rho\pd_t W -  \mu \Delta W + \nabla \Xi=0 &\qquad{\rm in}~\dBR^n, t>0,   \\
\dv W = 0 &\qquad {\rm in}~\dBR^n, t>0, \\
\jump{S(W,\Xi)\nu}=\jump{G-\mu D(\Psi)\nu}&\qquad{\rm on}~\bHS, t>0, \\
\jump{W}=\jump{H-\Psi}&\qquad{\rm on}~\BR^n_0, t>0, \\
W|_{t=0} = 0&\qquad{\rm in}~\BR^{n-1}. 
\end{aligned}\right.
\end{equation}
Let 
\begin{align*}
&\tg:=g-\mu D(\psi)\nu, \quad \th:=h-\psi,\\
&\tG:=G-\mu D(\Psi)\nu, \quad \tH:=H-\Psi. 
\end{align*}
Since we have the estimates  
\begin{align*}
&\|(|\lambda|^{1/2}\tg, \nabla \tg, |\lambda|\th, \nabla^2\th, |\lambda||\nabla'|^{-1}\pd_n \th_n)\|_{L_q(\dBR^n)} \\
\le &C\|(f, |\lambda|^{1/2}g, \nabla g, |\lambda|h, \nabla^2 h, |\lambda||\nabla'|^{-1}\pd_n h_n)\|_{L_q(\dBR^n)}, \\
&\|e^{-\gamma t}(\Lambda_\gamma^{1/2} \tG, \nabla \tG, \pd_t \tH, \nabla^2\tH, \pd_t(|\nabla'|^{-1}\pd_n \tH_n))\|_{L_p(\BR, L_q(\dBR^n))} \\
\le &C\|e^{-\gamma t}(F, \Lambda_\gamma^{1/2} G, \nabla G, \pd_t H, \nabla^2 H, \pd_t(|\nabla'|^{-1}\pd_n H_n))\|_{L_p(\BR, L_q(\dBR^n))}, 
\end{align*}
we conclude that $f=f_d=0$ and $F=F_d=0$ are enough to consider in theorems \ref{resolventthm} and \ref{maxregthm}, where we prove 
\begin{align*}
\||\lambda||\nabla'|^{-1}\pd_n \psi_n\|_{L_q(\dBR^n)} &\le C \|f\|_{L_q(\dBR^n)}\\
\|e^{-\gamma t}\pd_t (|\nabla'|^{-1}\pd_n \Psi_n)\|_{L_p(\BR, L_q(\dBR^n))} &\le C \|e^{-\gamma t} F\|_{L_p(\BR, L_q(\dBR^n))}
\end{align*}
in \ref{appA}. 

\section{Solution formula for the problems without $f$, $f_d$ and surface tension}\label{formula}

We give a solution of the resolvent problem \eqref{RSwithout} with $f=f_d=0$ and $\lambda\in \Sigma_\eps$.  
We apply partial Fourier transform with respect to tangential direction $x'\in\BR^{n-1}$. 
We use the notations 
\begin{align*}
\hv(\xi', x_n):= &\CF_{x'}v(\xi', x_n):=\int_{\BR^{n-1}} e^{-ix'\cdot\xi'} v(x', x_n)dx', \\
&\CF^{-1}_{\xi'}w(x',x_n) = \frac{1}{(2\pi)^{n-1}} \int_{\BR^{n-1}} e^{ix'\cdot\xi'} w(\xi', x_n)d\xi' 
\end{align*}
for functions $v, w:\BR^n_\pm \to \BC$. 
Let $u_\pm = {}^t(u_{\pm1}, \ldots, u_{\pm (n-1)}, u_{\pm n})$. 
Here and here after the index $j$ runs from $1$ to $n-1$ if we do not indicate. 

We need to solve the following second order ordinary differential equations; 
\begin{equation}\label{ODE}\left\{
\begin{aligned}
(\rho_\pm \lambda + \mu_\pm |\xi'|^2- \mu_\pm \pd_n^2)\hu_{\pm j} + i\xi_j \htheta_\pm =0\quad &\text{in}~ x_n\neq 0, \\
(\rho_\pm \lambda + \mu_\pm |\xi'|^2-\mu_\pm \pd_n^2)\hu_{\pm n} + \pd_n \htheta_\pm =0\quad &\text{in}~x_n\neq 0, \\
\sum_{j=1}^{n-1} i\xi_j\hu_{\pm j}  + \pd_n \hu_{\pm n}  = 0\quad &\text{in}~x_n\neq 0, \\
\jump{\mu(i\xi_j\hu_n + \pd_n \hu_j)} = - \jump{\hg_j}\quad &\text{on}~x_n = 0,\\
\jump{2\mu\pd_n \hu_n-\htheta} = -\jump{\hg_n}\quad &\text{on}~x_n=0, \\
\jump{\hu}=\jump{\hh}\quad &\text{on}~x_n=0. 
\end{aligned}\right.
\end{equation}
Set 
\[A:=\sqrt{\sum_{j=1}^{n-1}\xi_j^2}, \qquad  B_\pm := \sqrt{\rho_\pm (\mu_\pm)^{-1}\lambda + A^2}\]
with positive real parts. 
Here we consider $\xi'$ as complex values; 
\[\xi_j\in\tilde{\Sigma}_\eta:=\{z\in\BC\setminus\{0\}\mid |\arg z|<\eta\} \cup \{z\in\BC\setminus\{0\}\mid \pi-\eta<|\arg z|\}\]
for $\eta\in(0, \pi/4)$. 
The details are given in Lemma \ref{es}.  

We find the solution of the form 
\[\hu_{\pm j}(\xi', x_n) = \alpha_{\pm j}  (e^{\mp B x_n} - e^{\mp Ax_n}) + \beta_{\pm j} e^{\mp Bx_n} ~(j=1,\ldots, n), \qquad \htheta_\pm (\xi',x_n)=\gamma_\pm  e^{\mp A x_n}. \]
Then, the equations become 
\begin{equation}\label{ODE2}\left\{
\begin{aligned}
- \mu_\pm (B_\pm ^2 - A^2)\alpha_{\pm j} + i\xi_j \gamma_\pm= 0, \\
-\mu_\pm (B_\pm^2 - A^2)\alpha_{\pm n}  \mp  A\gamma_\pm = 0, \\
-i\xi'\cdot\alpha_\pm' \pm A\alpha_{\pm n}=0, \\
 i\xi'\cdot (\alpha_\pm' + \beta_\pm')\mp B_\pm (\alpha_{\pm n}+\beta_{\pm n}) =0, \\
\end{aligned}\right.
\end{equation}
and 
\begin{equation*}\left\{
\begin{aligned}
\mu_+(B_+^2-A^2)\alpha_{+n} + \mu_+ (B_+^2+A^2)\beta_{+n}-\mu_- (B_-^2-A^2)\alpha_{-n} - \mu_-(B_-^2+A^2)\beta_{-n} = i\xi'\cdot \jump{\hg'}, \\
\mu_+(B_+-A)^2\alpha_{+n}-2\mu_+AB_+\beta_{+n}+\mu_-(B_- - A)^2\alpha_{-n} -2\mu_- AB_- \beta_{-n} = -A\jump{\hg'_n}, \\
(B_+-A)\alpha_{+n} + B_+ \beta_{+n}+(B_--A)\alpha_{-n} + B_-\beta_{-n} = i\xi'\cdot\jump{\hh'}, \\
\beta_{+n}-\beta_{-n}=\jump{\hh'_n}. 
\end{aligned}\right.
\end{equation*}
This means 
\begin{align*}
&\begin{bmatrix}
\mu_+(B_++A)&\mu_+(B_+^2+A^2)&-\mu_-(B_-+A)&-\mu_-(B_-^2+A^2)\\
\mu_+(B_+-A)&-2\mu_+ AB_+ & \mu_-(B_--A) &-2\mu_-AB_-\\
1&B_+&1&B_-\\
0&1&0&-1\\
\end{bmatrix}
\begin{bmatrix}
(B_+-A)\alpha_{+n}\\
\beta_{+n}\\
(B_--A)\alpha_{-n}\\
\beta_{-n}
\end{bmatrix}\\
&=
\begin{bmatrix}
i\xi'\cdot\jump{\hg'}\\
-A\jump{\hg_n}\\
i\xi'\cdot\jump{\hh'}\\
\jump{\hh'_n}
\end{bmatrix}
= 
\begin{bmatrix}
i\xi'^T&0&0&0\\
0&-A&0&0\\
0&0&i\xi'^T&0\\
0&0&0&1
\end{bmatrix}
\begin{bmatrix}
\jump{\hg_1}\\
\vdots\\
\jump{\hg_n}\\
\jump{\hh_1}\\
\vdots\\
\jump{\hh_n}
\end{bmatrix}, 
\end{align*}
where $\xi'^T$ is the transpose of $\xi'$. 
We define 
\begin{align*}
&L:=\begin{bmatrix}
\mu_+(B_++A)&\mu_+(B_+^2+A^2)&-\mu_-(B_-+A)&-\mu_-(B_-^2+A^2)\\
\mu_+(B_+-A)&-2\mu_+ AB_+ & \mu_-(B_--A) &-2\mu_-AB_-\\
1&B_+&1&B_-\\
0&1&0&-1\\
\end{bmatrix}, \\
&R:=(r_{ij}):=\begin{bmatrix}
i\xi'^T&0&0&0\\
0&-A&0&0\\
0&0&i\xi'^T&0\\
0&0&0&1
\end{bmatrix}\quad (1\le i\le 4, 1\le j \le 2n). 
\end{align*}
We have 
\begin{align*}
\det{L}&= (\mu_+-\mu_-)^2 A^3 - \{(3\mu_+-\mu_-)\mu_+ B_+ + (3\mu_- - \mu_+)\mu_- B_-\} A^2 \\
&\qquad - \{(\mu_+B_+ + \mu_- B_-)^2 + \mu_+ \mu_- (B_++B_-)^2\}A - (\mu_+B_+ + \mu_-B_-)(\mu_+B_+^2 +\mu_- B_-^2)
\end{align*}
and it is known, in \cite[Lemma 5.5]{SS11}, that the determinant is not zero for $\lambda\in \Sigma_\eps$, $\xi'\in \BR^{n-1}$. 

We introduce the new notation 
\begin{align*}
&\CM_+=\CM_+(A, B_+, x_n)= \frac{e^{-B_+ x_n} - e^{-Ax_n}}{B_+-A}, \\
&\CM-=\CM_-(A, B_-, x_n)= \frac{e^{B_- x_n} - e^{Ax_n}}{B_--A}, \\
&(a_{i,j}):= (L^{-1}R)_{ij} =(\det{L})^{-1}(\sum_{s=1}^4 L_{is}r_{sj})_{ij}\quad (1\le i\le 4, 1\le j\le 2n), 
\end{align*}
where we use the cofactor matrix of $L$, denoted by ${\rm Cof}(L)=(L_{ij})$. 

From these observations, we have 
\begin{align*}
\hu_{+n}(\xi',x_n) &= \sum_{k=1}^n \left\{(a_{1,k}\CM_+ + a_{2,k} e^{-B_+x_n})\jump{\hg_k} + (a_{1,n+k}\CM_+ + a_{2,n+k} e^{-B_+x_n})\jump{\hh_k}\right\}, \\
\hu_{-n}(\xi',x_n) &= \sum_{k=1}^n \left\{(a_{3,k}\CM_- + a_{4,k} e^{B_-x_n})\jump{\hg_k} + (a_{3,n+k}\CM_- + a_{4,n+k} e^{B_-x_n})\jump{\hh_k}\right\}. 
\end{align*}

To simplify, we define the symbols; for $k=1,\ldots, n$, 
\begin{alignat*}{3}
&\phi_{k,+n}(\lambda, \xi', x_n)=a_{1,k}\CM_+ + a_{2,k} e^{-B_+ x_n}, \\
&\psi_{k,+n}(\lambda, \xi', x_n)=a_{1,n+k}\CM_+ + a_{2,n+k}e^{-B_+ x_n}, \\
&\phi_{k,-n}(\lambda, \xi', x_n)=a_{3,k}\CM_- + a_{4,k}e^{B_- x_n}, \\
&\psi_{k,-n}(\lambda, \xi', x_n)=a_{3,n+k}\CM_- + a_{4,n+k} e^{B_- x_n}, 
\end{alignat*}
which derives the solution formula for $u_{\pm n}$; 
\begin{align*}
\hu_{+n}=\hu_{+n}(\xi',x_n) &=\sum_{k=1}^n (\phi_{k,+n}\jump{\hg_k} + \psi_{k,+n} \jump{\hh_k}), \quad x_n>0, \\
 \hu_{-n}=\hu_{-n}(\xi',x_n) &=\sum_{k=1}^n (\phi_{k,-n}\jump{\hg_k} + \psi_{k,-n} \jump{\hh_k}), \quad x_n<0.  
\end{align*}

Since 
\begin{align*}
\gamma_{\pm} = \mp \frac{\mu_\pm (B_\pm +A)}{A} (B_\pm -A)\alpha_{\pm n}
\end{align*}
from the second equation of \eqref{ODE2}, 
by letting 
\begin{align*}
&\chi_{k,+}(\lambda, \xi', x_n) = -  \frac{\mu_+(B_++A)}{A}a_{1,k}e^{-Ax_n}, \\
&\omega_{k,+}(\lambda, \xi', x_n) = -  \frac{\mu_+(B_++A)}{A}a_{1,n+k}e^{-Ax_n},\\
&\chi_{k,-}(\lambda, \xi', x_n) =  \frac{\mu_-(B_-+A)}{A}a_{3,k} e^{Ax_n},\\
&\omega_{k,-}(\lambda, \xi', x_n) = \frac{\mu_-(B_-+A)}{A}a_{3,n+k} e^{Ax_n},
\end{align*}
we have 
\begin{align*}
\htheta_\pm = \htheta_\pm(\xi', x_n) = \sum_{k=1}^n (\chi_{k,\pm} \jump{\hg_k} + \omega_{k,\pm}\jump{\hh_k}), \quad x_n\gtrless0. 
\end{align*}

From the first equation of \eqref{ODE2}, $\alpha_{\pm j}=\mp (i\xi_j/A) \alpha_{\pm n}$. 
From the fourth and the sixth equations of \eqref{ODE}, $\beta_{\pm j}$ satisfy 
\begin{align*}
\begin{bmatrix}
\mu_+ B_+&\mu_- B_-\\
1 & -1
\end{bmatrix}
\begin{bmatrix}
\beta_{+j}\\
\beta_{-j}
\end{bmatrix}
&=\begin{bmatrix}\jump{\hg_j}\\
\jump{\hh_j}
\end{bmatrix}
+ 
\begin{bmatrix}
 - \mu_+(B_+-A)\alpha_{+j} + \mu_+i\xi_j \beta_{+n}- \mu_-(B_- -A)\alpha_{-j}  - \mu_- i\xi_j\beta_{-n}\\
 0
\end{bmatrix}\\
&=\begin{bmatrix}
\jump{\hg_j}\\
\jump{\hh_j}
\end{bmatrix}
+ \frac{i\xi_j}{A}
\begin{bmatrix}
\mu_+ &\mu_+ A  & - \mu_-  & - \mu_- A \\
0&0&0&0
\end{bmatrix}
\begin{bmatrix}
(B_+-A)\alpha_{+n}\\
\beta_{+n}\\
(B_--A)\alpha_{-n}\\
\beta_{-n}
\end{bmatrix}, 
\end{align*}
\begin{align*}
&\begin{bmatrix}
\beta_{+j}\\
\beta_{-j}
\end{bmatrix}
=
\frac{1}{\mu_+ B_+ + \mu_- B_-}\left\{ 
\begin{bmatrix}
\jump{\hg_j} + \mu_- B_- \jump{\hh_j}\\
\jump{\hg_j} - \mu_+ B_+ \jump{\hh_j}
\end{bmatrix}
+ 
\frac{i\xi_j}{A} 
\begin{bmatrix}
\mu_+ &\mu_+ A& - \mu_- & - \mu_- A\\ 
\mu_+ &\mu_+ A& - \mu_- & - \mu_- A
\end{bmatrix}
\begin{bmatrix}
(B_+-A)\alpha_{+n}\\
\beta_{+n}\\
(B_--A)\alpha_{-n}\\
\beta_{-n}
\end{bmatrix}
\right\}\\
&=\frac{1}{\mu_+ B_+ + \mu_- B_-}
\begin{bmatrix}
\jump{\hg_j} + \mu_- B_- \jump{\hh_j}+ \frac{i\xi_j}{A} \sum_{k=1}^n \{\mu_+(a_{1,k}+A a_{2,k})\jump{\hg_k} - \mu_-(a_{3,n+k}+Aa_{4,n+k})\jump{\hh_k}\}\\
\jump{\hg_j} - \mu_+ B_+ \jump{\hh_j} + \frac{i\xi_j}{A} \sum_{k=1}^n \{\mu_+(a_{1,k}+A a_{2,k})\jump{\hg_k} - \mu_-(a_{3,n+k}+Aa_{4,n+k})\jump{\hh_k}\} 
\end{bmatrix}. 
\end{align*}
Therefore 
\begin{align*}
\hu_{\pm j}=\hu_{\pm j}(\xi', x_n)&=(B_\pm - A)\alpha_{\pm j} \CM_\pm + \beta_{\pm j}e^{\mp B_\pm x_n}\\
&=:\sum_{k=1}^n(\phi_{k,\pm j}\jump{\hg_k} + \psi_{k,\pm j}\jump{\hh_k}), \quad x_n \gtrless0, 
\end{align*}
where 
\begin{align*}
\phi_{k,+j}(\lambda, \xi', x_n)&=-\frac{i\xi_j}{A} a_{1,k}\CM_+ + \frac{1}{\mu_+ B_+ + \mu_- B_-}(\delta_{k,j}+ \frac{i\mu_+ \xi_j}{A}(a_{1,k}+Aa_{2,k}))e^{-B_+ x_n}, \\
\psi_{k,+j}(\lambda, \xi', x_n)&=-\frac{i\xi_j}{A} a_{1,n+k}\CM_+ +  \frac{1}{\mu_+ B_+ + \mu_- B_-}(\mu_- B_- \delta_{k,j} - \frac{i\mu_- \xi_j}{A} (a_{3,n+k} + Aa_{4,n+k}))e^{-B_+ x_n}, \\
\phi_{k,-j}(\lambda, \xi', x_n)&=\frac{i\xi_j}{A} a_{3,k}\CM_- + \frac{1}{\mu_+ B_+ + \mu_- B_-}(\delta_{k,j}+ \frac{i\mu_+ \xi_j}{A}(a_{1,k}+Aa_{2,k}))e^{B_- x_n}, \\
\psi_{k,-j}(\lambda, \xi', x_n)&=\frac{i\xi_j}{A} a_{3,n+k}\CM_- -  \frac{1}{\mu_+ B_+ + \mu_- B_-}(\mu_+ B_+ \delta_{k,j} + \frac{i\mu_- \xi_j}{A} (a_{3,n+k} + Aa_{4,n+k}))e^{B_- x_n}. 
\end{align*}

Here we introduce a new notation $\jjump{f}(x', x_n):=f(x', x_n) - f(x', -x_n)$ for $f:\dBR^n\to \BC$. 
Then we have $\jump{f}(x')=\lim_{x_n\to+0}\jjump{f}(x', x_n)$ and  $\jump{a}(\xi')=\mp \int_{\BR_\pm} \jjump{\pd_n a}(\xi',y_n) dy_n$ for a function $a$ with $a(\cdot, x_n)\to0$ as $x_n\to \pm\infty$. 
So, We obtain the solution formula; 
\begin{align}
u_{\pm j}(x)&=\mp \sum_{k=1}^n \left\{\int_{\BR_\pm} \CF_{\xi'}^{-1}\left[\pd_n\phi_{k, \pm j}(\lambda, \xi', x_n+y_n)\CF_{x'} \jjump{g_k}\right](x, y_n) dy_n \right. \nonumber\\
&\qquad \quad  + \int_{\BR_\pm} \CF_{\xi'}^{-1} \left[\phi_{k, \pm j}(\lambda, \xi', x_n+y_n)\CF_{x'} \jjump{\pd_n g_k}\right](x,y_n) dy_n\nonumber\\
&\qquad \quad + \int_{\BR_\pm} \CF_{\xi'}^{-1}\left[\pd_n\psi_{k, \pm j}(\lambda, \xi', x_n+y_n)\CF_{x'} \jjump{h_k}\right](x, y_n) dy_n\nonumber\\
&\qquad \quad \left. + \int_{\BR_\pm} \CF_{\xi'}^{-1} \left[\psi_{k, \pm j}(\lambda, \xi', x_n+y_n)\CF_{x'} \jjump{\pd_n h_k}\right](x,y_n) dy_n\right\},  \quad (j=1, \ldots, n), \label{sol1}\\
\theta_\pm (x)&= \mp \sum_{k=1}^n \left\{\int_{\BR_\pm}  \CF_{\xi'}^{-1}\left[\pd_n\chi_{k,\pm}(\lambda, \xi', x_n+y_n)\CF_{x'} \jjump{g_k}\right](x, y_n) dy_n \right.\nonumber\\
&\qquad \quad  + \int_{\BR_\pm} \CF_{\xi'}^{-1} \left[\chi_{k, \pm}(\lambda, \xi', x_n+y_n)\CF_{x'} \jjump{\pd_n g_k}\right](x,y_n) dy_n \nonumber\\
&\qquad \quad + \int_{\BR_\pm}  \CF_{\xi'}^{-1}\left[\pd_n\omega_{k,\pm}(\lambda, \xi', x_n+y_n)\CF_{x'} \jjump{h_k}\right](x, y_n) dy_n \nonumber\\
&\qquad \quad \left. + \int_{\BR_\pm} \CF_{\xi'}^{-1} \left[\omega_{k, \pm}(\lambda, \xi', x_n+y_n)\CF_{x'} \jjump{\pd_n h_k}\right](x,y_n) dy_n\right\}. \label{sol2}
\end{align}

Since Laplace transformed non-stationary Stokes equations \eqref{Swithout} with $F=F_d=0$ on $\BR$ are the resolvent problem \eqref{RSwithout} with $f=f_d=0$, we have the following formula; 
\begin{align}
U_{\pm j}(x, t)&=\mp \CL_{\lambda}^{-1}\sum_{k=1}^n \left\{\int_{\BR_\pm} \CF_{\xi'}^{-1}\left[\pd_n\phi_{k, \pm j}(\lambda, \xi', x_n+y_n)\CF_{x'} \CL\jjump{G_k}\right](x, y_n) dy_n \right.\nonumber\\
&\qquad \quad  + \int_{\BR_\pm} \CF_{\xi'}^{-1} \left[\phi_{k, \pm j}(\lambda, \xi', x_n+y_n)\CF_{x'}\CL \jjump{\pd_n G_k}\right](x,y_n) dy_n\nonumber\\
&\qquad \quad + \int_{\BR_\pm} \CF_{\xi'}^{-1}\left[\pd_n\psi_{k, \pm j}(\lambda, \xi', x_n+y_n)\CF_{x'} \CL\jjump{H_k}\right](x, y_n) dy_n\nonumber\\
&\qquad \quad \left. + \int_{\BR_\pm} \CF_{\xi'}^{-1} \left[\psi_{k, \pm j}(\lambda, \xi', x_n+y_n)\CF_{x'} \CL\jjump{\pd_n H_k}\right](x,y_n) dy_n\right\},  \quad (j=1, \ldots, n),\label{sol3}\\
\Theta_\pm (x, t)&= \mp\CL_\lambda^{-1} \sum_{k=1}^n \left\{\int_{\BR_\pm}  \CF_{\xi'}^{-1}\left[\pd_n\chi_{k,\pm}(\lambda, \xi', x_n+y_n)\CF_{x'}\CL \jjump{G_k}\right](x, y_n) dy_n \right.\nonumber\\
&\qquad \quad  + \int_{\BR_\pm} \CF_{\xi'}^{-1} \left[\chi_{k, \pm}(\lambda, \xi', x_n+y_n)\CF_{x'} \CL\jjump{\pd_n G_k}\right](x,y_n) dy_n \nonumber\\
&\qquad \quad + \int_{\BR_\pm}  \CF_{\xi'}^{-1}\left[\pd_n\omega_{k,\pm}(\lambda, \xi', x_n+y_n)\CF_{x'}\CL \jjump{H_k}\right](x, y_n) dy_n \nonumber\\
&\qquad \quad \left. + \int_{\BR_\pm} \CF_{\xi'}^{-1} \left[\omega_{k, \pm}(\lambda, \xi', x_n+y_n)\CF_{x'} \CL\jjump{\pd_n H_k}\right](x,y_n) dy_n\right\}. \label{sol4}
\end{align}

\section{Proof of estimates for the problem without surface tension and gravity}\label{proof}

We decompose the solution \eqref{sol1}, \eqref{sol2} so that the independent variables are the right-hand side of the resolvent estimates. 
Analysis of the solutions \eqref{sol3} and \eqref{sol4} are based on the analysis of corresponding resolvent equations. 

We prepare a theorem to prove the main theorems. 
Let us define the operators $T$ and $\tilde{T}_\gamma$ by 
\begin{align*}
T[m] f(x)&=\int_0^\infty [\CF^{-1}_{\xi'} m(\xi', x_n+y_n) \CF_{x'} f](x, y_n)dy_n, \\
\tilde{T}_\gamma[m_\lambda] g(x, t)&=\CL_\lambda^{-1} \int_0^\infty [\CF^{-1}_{\xi'} m_\lambda(\xi', x_n+y_n) \CF_{x'} \CL g](x,y_n)dy_n, \\
&=[e^{\gamma t}\CF^{-1}_{\tau\to t} T[m_\lambda] \CF_{t\to\tau} (e^{-\gamma t}g)](x, t), 
\end{align*}
where $\lambda=\gamma+i\tau\in\Sigma_\eps$, $m, m_\lambda:\uHS\to\BC$ are multipliers, and $f:\uHS\to\BC$ and $g:\BR\times\uHS\to\BC$. 
The following theorem has shown in the paper \cite{K22}. 

\begin{thm}\label{thm}
{\rm (i)}  
Let $m$ satisfy the following two conditions: \\
{\rm (a)} There exists $\eta\in(0,\pi/2)$ such that $\{m(\cdot, x_n), x_n>0\}\subset H^\infty(\tilde{\Sigma}_\eta^{n-1})$. \\      
{\rm (b)} There exist $\eta\in(0,\pi/2)$ and $C>0$ such that $\sup_{\xi'\in\tilde{\Sigma}^{n-1}_\eta}|m(\xi', x_n)|\le Cx_n^{-1}$ for all $x_n>0$. \\
Then $T[m]$ is a bounded linear operator on $L_q(\uHS)$ for every $1<q<\infty$. \\
{\rm (ii)} Let $\gamma_0\ge0$ and let $m_\lambda$ satisfy the following two conditions: \\
{\rm (c)} There exists $\eta\in(0,\pi/2-\eps)$ such that for each $x_n>0$ and $\gamma\ge\gamma_0$, 
\[\tilde{\Sigma}_\eta^n \ni (\tau, \xi') \mapsto m_\lambda(\xi',x_n) \in\BC\]
 is bounded and holomorphic. \\
{\rm (d)} There exist $\eta\in(0,\pi/2-\eps)$ and $C>0$ such that $\sup\{|m_\lambda(\xi', x_n)| \mid (\tau, \xi')\in\tilde{\Sigma}_\eta^n\}\le Cx_n^{-1}$ for all $\gamma\ge\gamma_0$ and $x_n>0$. \\
Then $\tilde{T}_\gamma[m_\lambda]$ satisfies  
\[\|e^{-\gamma t}\tilde{T}_\gamma g\|_{L_p(\BR, L_q(\uHS))} \le C \|e^{-\gamma t} g\|_{L_p(\BR, L_q(\uHS))}\]
for every $\gamma\ge\gamma_0$ and $1<p,q<\infty$. 
\end{thm}

By using the following identity
\[B_\pm^2=\rho_\pm (\mu_\pm)^{-1}\lambda+\sum_{m=1}^{n-1} \xi_m^2, \qquad 1=\frac{B_\pm^2}{B_\pm^2}=\frac{\rho_\pm (\mu_\pm)^{-1}\lambda^{1/2}}{B_\pm^2}\lambda^{1/2} - \sum_{m=1}^{n-1}\frac{i\xi_m}{B_\pm^2}(i\xi_m), \]
we have 
\begin{align*}
u_{\pm j}(x)&=\mp \sum_{k=1}^n \left\{\int_{\BR_\pm} \CF_{\xi'}^{-1}\left[\rho_\pm(\mu_\pm)^{-1}\lambda^{1/2}B_\pm^{-2}\pd_n\phi_{k, \pm j}(\lambda, \xi', x_n+y_n)\CF_{x'} \jjump{\lambda^{1/2}g_k}\right](x, y_n) dy_n \right. \\
&\qquad \qquad - \sum_{m=1}^{n-1} \int_{\BR_\pm} \CF_{\xi'}^{-1}\left[i\xi_m B_\pm^{-2} \pd_n\phi_{k, \pm j}(\lambda, \xi', x_n+y_n)\CF_{x'} \jjump{\pd_m g_k}\right](x, y_n) dy_n\\
&\qquad \qquad  + \int_{\BR_\pm} \CF_{\xi'}^{-1} \left[\phi_{k, \pm j}(\lambda, \xi', x_n+y_n)\CF_{x'} \jjump{\pd_n g_k}\right](x,y_n) dy_n\\
&\qquad \qquad + \int_{\BR_\pm} \CF_{\xi'}^{-1}\left[B_\pm^{-2}\pd_n\psi_{k, \pm j}(\lambda, \xi', x_n+y_n)\CF_{x'} \jjump{(\rho_\pm(\mu_\pm)^{-1}\lambda - \Delta')h_k}\right](x, y_n) dy_n\\
&\qquad \qquad  + \int_{\BR_\pm} \CF_{\xi'}^{-1} \left[\rho_\pm(\mu_\pm)^{-1}\lambda^{1/2}B_\pm^{-2}\psi_{k, \pm j}(\lambda, \xi', x_n+y_n)\CF_{x'} \jjump{\lambda^{1/2}\pd_n h_k}\right](x,y_n) dy_n \\
&\qquad \qquad \left.- \sum_{m=1}^{n-1} \int_{\BR_\pm} \CF_{\xi'}^{-1} \left[i\xi_m B_\pm^{-2}\psi_{k, \pm j}(\lambda, \xi', x_n+y_n)\CF_{x'} \jjump{\pd_m \pd_n h_k}\right](x,y_n) dy_n \right\} \quad (j=1, \ldots, n), \\
\theta_\pm(x)&=\mp \sum_{k=1}^n \left\{\int_{\BR_\pm} \CF_{\xi'}^{-1}\left[\rho_\pm(\mu_\pm)^{-1}\lambda^{1/2}B_\pm^{-2}\pd_n\chi_{k,\pm}(\lambda, \xi', x_n+y_n)\CF_{x'} \jjump{\lambda^{1/2}g_k}\right](x, y_n) dy_n \right. \\
&\qquad \qquad - \sum_{m=1}^{n-1} \int_{\BR_\pm} \CF_{\xi'}^{-1}\left[i\xi_m B_\pm^{-2} \pd_n\chi_{k, \pm}(\lambda, \xi', x_n+y_n)\CF_{x'} \jjump{\pd_m g_k}\right](x, y_n) dy_n\\
&\qquad \qquad  + \int_{\BR_\pm} \CF_{\xi'}^{-1} \left[\chi_{k, \pm}(\lambda, \xi', x_n+y_n)\CF_{x'} \jjump{\pd_n g_k}\right](x,y_n) dy_n\\
&\qquad \qquad + \int_{\BR_\pm} \CF_{\xi'}^{-1}\left[B_\pm^{-2}\pd_n\omega_{k, \pm}(\lambda, \xi', x_n+y_n)\CF_{x'} \jjump{(\rho_\pm(\mu_\pm)^{-1}\lambda - \Delta')h_k}\right](x, y_n) dy_n\\
&\qquad \qquad  + \int_{\BR_\pm} \CF_{\xi'}^{-1} \left[\rho_\pm(\mu_\pm)^{-1}\lambda^{1/2}B_\pm^{-2}\omega_{k, \pm}(\lambda, \xi', x_n+y_n)\CF_{x'} \jjump{\lambda^{1/2}\pd_n h_k}\right](x,y_n) dy_n \\
&\qquad \qquad \left.- \sum_{m=1}^{n-1} \int_{\BR_\pm} \CF_{\xi'}^{-1} \left[i\xi_m B_\pm^{-2}\omega_{k, \pm}(\lambda, \xi', x_n+y_n)\CF_{x'} \jjump{\pd_m \pd_n h_k}\right](x,y_n) dy_n \right\}. 
\end{align*}

Let $S_\pm^u(\lambda, \xi', x_n)$ and $S^\theta_\pm(\lambda, \xi', x_n)$ be any of symbols; 
\begin{align*}
&S_\pm^u(\lambda, \xi', x_n):=\begin{cases} \rho_\pm(\mu_\pm)^{-1}\lambda^{1/2}B_\pm^{-2}\pd_n\phi_{k, \pm j}(\lambda, \xi', x_n) & \text{or}, \\
i\xi_m B_\pm^{-2}\pd_n \phi_{k,\pm j}(\lambda, \xi', x_n) & \text{or}, \\
\phi_{k, \pm j}(\lambda, \xi', x_n) & \text{or}, \\
B_\pm^{-2} \pd_n \psi_{k, \pm j}(\lambda, \xi', x_n) & \text{or}, \\
\rho_\pm (\mu_\pm)^{-1} \lambda^{1/2}B_\pm^{-2}\psi_{k, \pm j}(\lambda, \xi', x_n)&\text{or}, \\
i\xi_m B_\pm^{-2}\psi_{k, \pm j}(\lambda, \xi', x_n), &\end{cases} \\
&S^\theta_\pm(\lambda, \xi', x_n):=\begin{cases} \rho_\pm(\mu_\pm)^{-1}\lambda^{1/2}B_\pm^{-2}\pd_n\chi_{k, \pm}(\lambda, \xi', x_n) & \text{or}, \\
i\xi_m B_\pm^{-2}\pd_n \chi_{k,\pm}(\lambda, \xi', x_n) & \text{or}, \\
\chi_{k, \pm}(\lambda, \xi', x_n) & \text{or}, \\
B_\pm^{-2} \pd_n \omega_{k, \pm}(\lambda, \xi', x_n) & \text{or}, \\
\rho_\pm (\mu_\pm)^{-1} \lambda^{1/2}B_\pm^{-2}\omega_{k, \pm}(\lambda, \xi', x_n)~(k\neq n)&\text{or}, \\
i\xi_m B_\pm^{-2}\omega_{k, \pm}(\lambda, \xi', x_n). &\end{cases}
\end{align*}
We shall prove that all of the symbols are bounded in the sense that 
\begin{align}
&\sup_{\substack{(\lambda, \xi')\in \Sigma_\eps\times \tilde{\Sigma}_\eta^{n-1}\\ \ell, \ell'=1, \ldots, n-1}}\left\{(|\lambda|+|\lambda|^{1/2}|\xi_\ell|+|\xi_\ell| |\xi_{\ell'}|)|S^u_\pm|+(|\lambda|^{1/2}+|\xi_\ell|)|\pd_n S^u_\pm| + |\pd_n^2 S^u_\pm| + |\xi_\ell| |S^\theta_\pm| + |\pd_n  S^\theta_\pm|\right\} \nonumber\\
&<C(\pm x_n)^{-1}\label{symbol}
\end{align}
for suitable $\eps, \eta$ so that we prove theorem \ref{resolventthm} with $f=f_d=0$. 

Following the paper \cite{K22}, we have some identities
\begin{align*}
\pd_n \CM_\pm (A, B_\pm ,x_n) &= \mp e^{\mp B_\pm x_n} \mp A\CM_\pm (A, B_\pm,x_n), \\
\pd^2_n  \CM_\pm (A, B_\pm ,x_n)  &= (A+B_\pm)e^{\mp B_\pm x_n} + A^2 \CM_\pm (A, B_\pm ,x_n) , \\
\pd^3_n \CM_\pm (A, B_\pm ,x_n) &=\mp (A^2+AB_\pm+B_\pm^2)e^{\mp B_\pm x_n} \mp A^3\CM_\pm (A, B_\pm ,x_n) 
\end{align*}
and an useful lemma; 
\begin{lem}\label{es}
Let $0<\eps<\pi/2$ and $0<\eta<\min\{\pi/4, \eps/2\}$. 
Then for any $(\lambda, \xi',x_n) \in \Sigma_\eps\times \tilde{\Sigma}_\eta^{n-1}\times \BR_\pm$, letting $A:=\sqrt{\sum_{j=1}^{n-1}\xi_j^2}$, $B_\pm:=\sqrt{\rho_\pm (\mu_\pm)^{-1}\lambda + A^2}$ and $\tA:=\sqrt{\sum_{j=1}^{n-1}|\xi_j|^2}$, we have 
\begin{align*}
c\tA & \le  \Re A \le |A| \le \tA,\\
c(|\lambda|^{1/2}+\tA)&\le \Re B_\pm \le |B_\pm| \le C(|\lambda|^{1/2} + \tA), \\
|\pd_n^m e^{\mp Ax_n}| & \le C\tA^{-1+m} (\pm x_n)^{-1}, ~{\rm for}~x_n\gtrless 0, \\
|\pd_n^m \CM_\pm(A, B_\pm, x_n)| &\le C(|\lambda|^{1/2}+\tA)^{-2+m}(\pm x_n)^{-1},  ~{\rm for}~x_n\gtrless 0\\
|\pd_n^m e^{\mp B_\pm x_n}| &\le C(|\lambda|^{1/2}+\tA)^{-1+m}(\pm x_n)^{-1}, ~{\rm for}~x_n\gtrless 0, \\
c(|\lambda|^{1/2}+\tA) &\le |\mu_+ B_+ + \mu_- B_-|
\end{align*}
for $m=0,1,2,3$, with positive constants $c$ and $C$, which are independent of $\lambda, \xi', x_n$. 
\end{lem}

We recall 
\begin{align*}
L^{-1}&=\begin{bmatrix}
\mu_+(B_++A)&\mu_+(B_+^2+A^2)&-\mu_-(B_-+A)&-\mu_-(B_-^2+A^2)\\
\mu_+(B_+-A)&-2\mu_+ AB_+ & \mu_-(B_--A) &-2\mu_-AB_-\\
1&B_+&1&B_-\\
0&1&0&-1\\
\end{bmatrix}^{-1}\\
&=(\det{L})^{-1} {\rm Cof}(L) = (\det{L})^{-1} (L_{ij})_{ij}. 
\end{align*}
From the cofactor expansion, we have 
\begin{align*}
|L_{is}| \le \begin{cases}
C(|\lambda|^{1/2} + \tA)~&{\rm for}~(i,s)=(2,1), (2,2), (4,1), (4,2), \\
C(|\lambda|^{1/2}  + \tA)^2~&{\rm for}~(i,s)=(1,1), (1,2), (2,3), (3,1), (3,2), (4,3), \\
C(|\lambda|^{1/2}  + \tA)^3~&{\rm for}~(i,s)=(1,3), (2,4), (3,3), (4,4), \\
C(|\lambda|^{1/2}  + \tA)^4~&{\rm for}~(i,s)=(1,4), (3,4), 
\end{cases}
\end{align*}
and then 
\begin{align*}
|\sum_{s=1}^4 L_{is}r_{sj}| \le \begin{cases}
C\tA(|\lambda|^{1/2} + \tA)~&{\rm for}~(i,j)=(2,1), \ldots, (2,n), (4,1), \ldots, (4,n), \\
C\tA(|\lambda|^{1/2}  + \tA)^2~&{\rm for}~(i,j)=\begin{cases}(1,1), \ldots, (1,n), (2,n+1), \ldots, (2,2n-1), \\(3, 1), \ldots, (3, n), (4,n+1), \ldots, (4, 2n-1), \end{cases}\\
C\tA(|\lambda|^{1/2}  + \tA)^3~&{\rm for}~(i,j)=(1,n+1), \ldots, (1, 2n-1), (3,n+1), \ldots, (3, 2n-1),  \\
C(|\lambda|^{1/2}  + \tA)^3~&{\rm for}~(i,j)=(2,2n), (4,2n), \\
C(|\lambda|^{1/2}  + \tA)^4~&{\rm for}~(i,j)=(1,2n), (3,2n). 
\end{cases}
\end{align*}

We need to prepare a lower estimate of $\det{L}$. 

\begin{lem}\label{es2}
Let $0<\eps< \pi/2$ and $0<\eta < \min\{\pi/4, \eps/2\}$. 
Then there exists a positive constant $c$ such that 
\[ c(|\lambda|^{1/2} + \tA)^3 \le |\det{L}|\qquad (\lambda\in \Sigma_\eps, \xi'\in \tilde{\Sigma}_\eta^{n-1}). \]
\end{lem}
\begin{proof}
The proof is almost same as the paper \cite[Lemma 5.5]{SS03} in which they proved for $\xi'\in\BR^{n-1}\setminus\{0\}$. 
Suitable change between $A$ and $\tA$ derives desired estimates. 
Note that $(|\lambda|^{1/2}+\tA)^3$ and $(|\lambda|+\tA^2)^{3/2}$ are equivalent. 
\end{proof}

We calculate the estimates of $\phi_{k, \pm j}$ and $\psi_{k, \pm j}$ by combining all estimates above; 
\begin{align*}
|\pd_n^m \phi_{k, \pm n}| &\le |a_{1,k}||\pd_n^m \CM_\pm| + |a_{2,k}| |\pd_n^m e^{\mp B_\pm x_n}|\\
&\le |\det{L}|^{-1} (|\sum_{s=1}^4 L_{1s}r_{sk}| |\pd_n^m \CM_\pm| + |\sum_{s=1}^4 L_{2s}r_{sk}| | |\pd_n^m e^{\mp B_\pm x_n}|)\\
&\le C(|\lambda|^{1/2}+\tA)^{-2+m} (\pm x_n)^{-1}, \\
|\pd_n^m \phi_{k, \pm j}| &\le C (|a_{1,k}||\pd_n^m \CM_\pm| + |\mu_+ B_+ + \mu_- B_-|^{-1} (1+ |a_{1,k}|+\tA|a_{2,k}|) |\pd_n^m e^{\mp B_\pm x_n}|)\\
&\le C(|\lambda|^{1/2}+\tA)^{-2+m} (\pm x_n)^{-1}, \\
\end{align*}
similarly, 
\begin{align*}
|\pd_n^m \psi_{k, \pm n}| &\le C(|\lambda|^{1/2}+\tA)^{-1+m} (\pm x_n)^{-1}, \\
|\pd_n^m \psi_{k, \pm j}| &\le C(|\lambda|^{1/2}+\tA)^{-1+m} (\pm x_n)^{-1}
\end{align*}
for $m=0,1,2,3$ and $k=1, \ldots, n$. 
On the other hand we have to take care of $\chi_{k, \pm}$ and $\omega_{k, \pm}$ carefully; 
\begin{align*}
|\pd_n^m \chi_{k, +}| &\le C (|\lambda|^{1/2}+\tA) \tA^{-1}|a_{1,k}| |\pd_n^m e^{-Ax_n}|\\
&\le C(|\lambda|^{1/2}+\tA) \tA^{-1} \cdot |\det{L}|^{-1} |\sum_{s=1}^4 L_{1s}r_{sk}| \cdot \tA^{m-1} x_n^{-1}\\
&\le C(|\lambda|^{1/2}+\tA)\tA^{-1} \cdot (|\lambda|^{1/2}+\tA)^{-3}\cdot \tA(|\lambda|^{1/2}+\tA)^2 \cdot \tA^{-1+m} x_n^{-1} \\
&\le C \tA^{-1+m}x_n^{-1}, 
\end{align*}
and similarly, 
\begin{align*}
|\pd_n^m \chi_{k, -}| &\le C \tA^{-1+m}(-x_n)^{-1}, \\
|\pd_n^m \omega_{k, \pm}| &\le C (|\lambda|^{1/2}+\tA) \tA^{-1+m} (\pm x_n)^{-1}~(k\neq n)
\end{align*}
for $m=0,1,2$. 
We remark that we cannot expect good estimate for $|\pd_n^m \omega_{n, \pm}|$ like above. 
We decompose as follows for the excluded term; 
\begin{align*}
&\int_{\BR_\pm} \CF_{\xi'}^{-1} \left[\rho_\pm(\mu_\pm)^{-1}\lambda^{1/2}B_\pm^{-2}\omega_{n, \pm}(\lambda, \xi', x_n+y_n)\CF_{x'} \jjump{\lambda^{1/2}\pd_n h_n}\right](x,y_n) dy_n\\
=& \int_{\BR_\pm} \CF_{\xi'}^{-1} \left[\rho_\pm(\mu_\pm)^{-1}A B_\pm^{-2}\omega_{n, \pm}(\lambda, \xi', x_n+y_n)\CF_{x'} \jjump{\lambda |\nabla'|^{-1}\pd_n h_n}\right](x,y_n) dy_n
\end{align*}
Since 
\begin{align*}
|\pd_n^m \omega_{n, +}| &\le C (|\lambda|^{1/2}+\tA)\tA^{-1} |a_{1,2n}||\pd_n^m e^{-Ax_n}|\\
&\le C(|\lambda|^{1/2}+\tA)\tA^{-1} \cdot (|\lambda|^{1/2}+\tA)^{-3}\cdot (|\lambda|^{1/2}+\tA)^4 \cdot \tA^{-1+m} x_n^{-1} \\
&\le C (|\lambda|^{1/2}+\tA)^2 \tA^{-2+m}x_n^{-1}~(x_n>0), \\
|\pd_n^m \omega_{n, -}| &\le C (|\lambda|^{1/2}+\tA)^2 \tA^{-2+m}(-x_n)^{-1}~(x_n<0), 
\end{align*}
we have 
\begin{align*}
\sup_{\substack{(\lambda, \xi')\in \Sigma_\eps\times \tilde{\Sigma}_\eta^{n-1}\\ \ell=1, \ldots, n-1}}\left\{|\xi_\ell| |A B_\pm^{-2}\omega_{n, \pm}(\lambda, \xi', x_n)| + |\pd_n(A B_\pm^{-2}\omega_{n, \pm}(\lambda, \xi', x_n))|\right\}<C(\pm x_n)^{-1}. 
\end{align*}

The inequality \eqref{symbol} with above discussion corresponds to the estimates $|\lambda| u$, $|\lambda|^{1/2}\pd_\ell u$, $\pd_\ell \pd_{\ell'} u$, $|\lambda|^{1/2}\pd_n u$, $\pd_\ell \pd_n u$, $\pd_n^2 u$, and $\pd_\ell \theta$ and $\pd_n\theta$ respectively. 

We also see that the new symbols $S_\pm^\theta$, $S_\pm^\theta$ and $AB_\pm^{-2}\omega_{n, \pm}$, multiplied by $\lambda$, $\xi_\ell$ and $\pd_n$, are holomorphic in $(\tau, \xi')\in\tilde{\Sigma}_\eta^n$. 
Therefore we are able to use theorem \ref{thm}, where we use change of variables from $x_n$ to $-x_n$, and $\|\jjump{f}\|_{L_q(\uHS)}\le \|f\|_{L_q(\dBR^n)}$ too. 

\begin{thm}
Let $0<\eps<\pi/2$ and $1<q<\infty$. 
Then for any $\lambda \in \Sigma_\eps, g\in W^1_q(\dBR^n)$, $h\in W^2_q(\dBR^n)$ and $h_n\in E_q(\dBR^n)$, the problem \eqref{RSwithout} with $f=f_d=0$ admits a solution $(u, \theta) \in W^2_q(\dBR^n) \times \hW^1_q(\dBR^n)$ with the resolvent estimate; 
\begin{align*}
\|(|\lambda|u, |\lambda|^{1/2}\nabla u, \nabla^2 u, \nabla \theta)\|_{L_q(\dBR^n)} 
&\le
C\|(|\lambda|^{1/2}g, \nabla g, |\lambda|h, |\lambda|^{1/2} \nabla h, \nabla^2 h, |\lambda||\nabla'|^{-1}\pd_n h_n) \|_{L_q(\dBR^n)}\\
&\le
C\|(|\lambda|^{1/2}g, \nabla g, |\lambda|h, \nabla^2 h, |\lambda||\nabla'|^{-1}\pd_n h_n) \|_{L_q(\dBR^n)}
\end{align*}
for some constant $C$. 
\end{thm}

This theorem and the estimates in section \ref{reduction} derives the existence part of theorem \ref{resolventthm}. 
The uniqueness has shown before where they considered the homogeneous equation and the dual problem. 

For the non-stationary Stokes equations we have the following theorem by theorem \ref{thm}; 
\begin{thm}
Let $1<p, q<\infty$ and $\gamma_0\ge 0$. 
Then for any 
\begin{align*}
H&\in W^1_{p,0,\gamma_0}(\BR, L_q(\dBR^n)) \cap L_{p,0,\gamma_0}(\BR, W^2_q(\dBR^n)), \\
H_n&\in W^1_{p,0,\gamma_0}(\BR, E_q(\dBR^n)), 
\end{align*}
the problem \eqref{non-stationary Stokes} with $F=F_d=0$ admits a solution $(U, \Pi)$ such that 
\begin{align*}
U&\in W^1_{p,0,\gamma_0}(\BR, L_q(\dBR^n)) \cap L_{p,0,\gamma_0}(\BR, W^2_q(\dBR^n)), \\
\Pi &\in L_{p,0,\gamma_0}(\BR, \hW^1_q(\dBR^n))
\end{align*} 
with the maximal $L_p$-$L_q$ regularity; 
\begin{align*}
\|e^{-\gamma t}(\pd_t U, \gamma U, \Lambda^{1/2}_\gamma \nabla U, \nabla^2 U, \nabla \Pi)\|_{L_p(\BR, L_q(\dBR^n))}
&\le
C\|e^{-\gamma t} (\pd_t H, \Lambda^{1/2}_\gamma \nabla H, \nabla^2 H, \pd_t(|\nabla'|^{-1}\pd_n H_n)) \|_{L_p(\BR, L_q(\dBR^n))}\\
&\le 
C\|e^{-\gamma t} (\pd_t H, \nabla^2 H, \pd_t(|\nabla'|^{-1}\pd_n H_n)) \|_{L_p(\BR, L_q(\dBR^n))}
\end{align*}
for any $\gamma \ge \gamma_0$ with some constant $C=C_{n, p, q, \gamma_0}$ depending only on $n, p, q$ and $\gamma_0$. 
\end{thm}

\section{On the problems with surface tension and gravity}\label{final}
In this section we consider the problems \eqref{Swith} and \eqref{RSwith} and prove theorems \ref{maxregthm2} and \ref{maxregthm2}. 
Let $(v, \tau)$ and $(V, \Upsilon)$ be solutions to the problems 
\begin{equation}
\left\{\begin{aligned}
\rho\lambda v -  \DV S(v,\tau)= f &\qquad{\rm in}~\dBR^n,   \\
\dv u = f_d &\qquad {\rm in}~\dBR^n, \\
\jump{S(v,\tau)\nu}=\jump{g} &\qquad{\rm on}~\bHS, \\
\jump{v}=\jump{h}&\qquad{\rm on}~\BR^n_0, 
\end{aligned}\right.
\end{equation}

\begin{equation}
\left\{\begin{aligned}
\rho\pd_t V -  \DV S(V,\Upsilon)=F &\qquad{\rm in}~\dBR^n, t>0,   \\
\dv V = F_d &\qquad {\rm in}~\dBR^n, t>0, \\
\jump{S(V,\Upsilon)\nu}=\jump{G} &\qquad{\rm on}~\bHS, t>0, \\
\jump{V}=\jump{H}&\qquad{\rm on}~\BR^n_0, t>0, \\
V|_{t=0} = 0&\qquad{\rm in}~\BR^{n-1}, 
\end{aligned}\right.
\end{equation}

We shall find the solutions $(w, \kappa)$ and $(W, \Xi)$ satisfying 
\begin{equation}\label{eq w}
\left\{\begin{aligned}
\rho\lambda w -  \DV S(w,\kappa)=0 &\qquad{\rm in}~\dBR^n,   \\
\dv w = 0 &\qquad {\rm in}~\dBR^n, \\
\lambda  \eta + w_n = d-v_n=: \td &\qquad{\rm on}~\BR^n_0, \\
\jump{S(w,\kappa)\nu}-(\jump{\rho}c_g + c_\sigma \Delta')\eta\nu=0 &\qquad{\rm on}~\bHS, \\
\jump{w}=0&\qquad{\rm on}~\BR^n_0, 
\end{aligned}\right.
\end{equation}
\begin{equation}\label{eq W}
\left\{\begin{aligned}
\rho\pd_t W -  \DV S(W,\Xi)=0 &\qquad{\rm in}~\dBR^n, t>0,   \\
\dv W = 0 &\qquad {\rm in}~\dBR^n, t>0, \\
\pd_t Y + W_n = D-V_n=:\tD &\qquad{\rm on}~\BR^n_0, t>0, \\
\jump{S(V,\Xi)\nu}-(\jump{\rho}c_g + c_\sigma \Delta')Y\nu=0 &\qquad{\rm on}~\bHS, t>0, \\
\jump{V}=0&\qquad{\rm on}~\BR^n_0, t>0, \\
(V,Y)|_{t=0} = (0,0)&\qquad{\rm in}~\dBR^n, 
\end{aligned}\right.
\end{equation}
then $(u, \theta)=(v+w, \tau+\kappa)$ and $(U, \Theta)=(V+W, \Upsilon+\Xi)$ are the solutions of \eqref{RSwith} and \eqref{Swith}. 
To solve the equations \eqref{eq w}, it is enough to consider 
\begin{align*}
(\jump{\hh}, \jump{\hg'}, \jump{\hg_n})=(0,0, -(\jump{\rho}c_g-c_\sigma A^2)\heta)
\end{align*}
in \eqref{ODE}, and 
\begin{equation}
\left\{\begin{aligned}
\lambda \heta + \hat{w}_n = \hat{\td} &\qquad{\rm on}~\BR^n_0, \\
\hat{w}_{\pm j}= \phi_{n,\pm j}\jump{\hg_n}&\qquad{\rm in}~\BR^n_\pm~(j=1, \cdots, n). 
\end{aligned}\right.
\end{equation}
Note that $\phi_{n, + n}(\lambda, \xi',0)=\phi_{n, -n}(\lambda, \xi', 0)=(\det{L})^{-1}A\{\mu_+(B_++A)+\mu_-(B_-+A)\}$. 
Therefore we have the following solution formulas; 
\begin{align*}
&\heta(\lambda, \xi', x_n)=\frac{\det{L}}{\lambda\det{L}-A\{\mu_+(B_++A)+\mu_-(B_-+A)(\jump{\rho}c_g-c_\sigma A^2)\}} \hat{\td}, \\
&\hat{w}_{\pm j}(\lambda, \xi', x_n)=  -\phi_{n,\pm j}(\jump{\rho}c_g-c_\sigma A^2)\heta\qquad~(j=1, \cdots, n), \\
&\hat{\kappa}_\pm (\lambda, \xi', x_n) = -\chi_{n, \pm}(\jump{\rho}c_g-c_\sigma A^2)\heta
\end{align*}
with an estimate 
\begin{align*}
\CL(\lambda, \xi')&:=\lambda\det{L}-A\{\mu_+(B_++A)+\mu_-(B_-+A)(\jump{\rho}c_g-c_\sigma A^2)\} \\
|\CL(\lambda, \xi')|&\ge c(|\lambda|+\tA)(|\lambda|^{1/2}+\tA)^3
\end{align*}
for $(\lambda, \xi')\in \Sigma_{\eps, \gamma_0}\times \tilde{\Sigma}_\eta^{n-1}$ with $0<\eps<\pi/2$, $0<\eta < \min\{\pi/4, \eps/2\}$ and $\gamma_0\ge 1$. 
The proof for $\xi'\in \BR^{n-1}$ is in the paper \cite[Lemma 6.1]{SS11}. 
However the proof for complex value is almost same. 

Since we have the estimate
\begin{align*}
&\sup_{\substack{(\lambda, \xi')\in \Sigma_{\eps, \gamma_0}\times \tilde{\Sigma}_\eta^{n-1}\\ \ell=1, \ldots, n-1}}\left\{(|\lambda|+|\xi_\ell|)\frac{\det{L}}{\CL}\right\}<C
\end{align*}
and holomorphy, we are able to prove, by Fourier multiplier theorem in \cite[Proposition 4.3.10, Theorem 4.3.3]{PS16}, 
\begin{align*}
\|(|\lambda| \eta, \nabla \eta)\|_{L_q(\dBR^n)}&\le C\|\td\|_{L_q(\dBR^n)}, \\
\|(|\lambda| \nabla \eta, \nabla^2 \eta)\|_{L_q(\dBR^n)}&\le C\|\nabla \td\|_{L_q(\dBR^n)}, \\
\|(|\lambda| \nabla^2 \eta, \nabla^3 \eta)\|_{L_q(\dBR^n)}&\le C\|\nabla^2 \td\|_{L_q(\dBR^n)}, 
\end{align*}
for $\lambda\in \Sigma_{\eps, \gamma_0}$. 
And then, from the results on previous section, 
\begin{align*}
\|(|\lambda| w, |\lambda|^{1/2}\nabla w, \nabla^2 w, \nabla \kappa)\|_{L_q(\dBR^n)} &\le C\|(|\lambda|^{1/2}g_n, \nabla g_n)\|_{L_q(\dBR^n)}\\
&\le C \|\eta\|_{W^3_q(\dBR^n)}\\
&\le C \| \td\|_{W^2_q(\dBR^n)}, 
\end{align*}
where we have used $|\lambda|^{1/2}\le |\lambda|$ when $\lambda\in \Sigma_{\eps, \gamma_0}$, and $C$ depends on $\gamma_0$. 
This concludes that 
\begin{align*}
&\|(|\lambda|u, |\lambda|^{1/2}\nabla u, \nabla^2 u, \nabla \theta)\|_{L_q(\dBR^n)} + |\lambda|\|\eta\|_{W^2_q(\dBR^n)} + \|\eta\|_{W^3_q(\dBR^n)} \\
\le 
&\|(|\lambda|v, |\lambda|^{1/2}\nabla v, \nabla^2 v, |\lambda|w, |\lambda|^{1/2}\nabla w, \nabla^2 w, \nabla \theta, \nabla \kappa)\|_{L_q(\dBR^n)} + |\lambda|\|\eta\|_{W^2_q(\dBR^n)} + \|\eta\|_{W^3_q(\dBR^n)} \\
\le 
&C_{n, q, \eps, \gamma_0}\left\{\|(f, |\lambda|^{1/2} f_d, \nabla f_d, |\lambda|^{1/2} g, \nabla g, |\lambda|h, \nabla^2 h, |\lambda||\nabla'|^{-1}\pd_n h_n) \|_{L_q(\dBR^n)} + |\lambda| \|f_d\|_{\hW^{-1}_q(\dBR^n)} + \|\td\|_{W^2_q(\dBR^n)}\right\}\\ 
\le 
&C_{n, q, \eps, \gamma_0}\left\{\|(f, |\lambda|^{1/2} f_d, \nabla f_d, |\lambda|^{1/2} g, \nabla g, |\lambda|h, \nabla^2 h, |\lambda||\nabla'|^{-1}\pd_n h_n) \|_{L_q(\dBR^n)} + |\lambda| \|f_d\|_{\hW^{-1}_q(\dBR^n)} + \|d\|_{W^2_q(\dBR^n)}\right\}
\end{align*}
since 
\begin{align*}
&\|\td\|_{W^2_q(\dBR^n)} \\
&\le C_{n, q, \eps} (\|d\|_{W^2_q(\dBR^n)} + \|v\|_{W^2_q(\dBR^n)})\\
&\le C_{n, q, \eps, \gamma_0}(\|d\|_{W^2_q(\dBR^n)} + |\lambda|\|v\|_{L_q(\dBR^n)} + \|\nabla^2 v\|_{L_q(\dBR^n)})\\
&\le  C_{n, q, \eps, \gamma_0}(\|d\|_{W^2_q(\dBR^n)} + \|(f, |\lambda|^{1/2} f_d, \nabla f_d, |\lambda|^{1/2} g, \nabla g, |\lambda|h, \nabla^2 h, |\lambda||\nabla'|^{-1}\pd_n h_n) \|_{L_q(\dBR^n)} + |\lambda| \|f_d\|_{\hW^{-1}_q(\dBR^n)}). 
\end{align*}
In addition, we have 
\begin{align*}
|\lambda|^{3/2}\|\eta\|_{W^1_q(\dBR^n)} &\le |\lambda|^{1/2}\|\td\|_{W^1_q(\dBR^n)}\\
&\le |\lambda|^{1/2}\|d\|_{W^1_q(\dBR^n)} + |\lambda|^{1/2}\|v\|_{W^1_q(\dBR^n)}\\
&\le C_{n, q, \eps,\gamma_0}\left\{\|(f, |\lambda|^{1/2} f_d, \nabla f_d, |\lambda|^{1/2} g, \nabla g, |\lambda|h, \nabla^2 h, |\lambda||\nabla'|^{-1}\pd_n h_n) \|_{L_q(\dBR^n)} \right.\\
&\left. \qquad \qquad \qquad + |\lambda| \|g\|_{\hW^{-1}_q(\dBR^n)} + \|d\|_{W^2_q(\dBR^n)} + |\lambda|^{1/2}\|d\|_{W^1_q(\dBR^n)}\right\}\end{align*}
and 
\begin{align*}
|\lambda|^{2}\|\eta\|_{L_q(\dBR^n)} &\le |\lambda|\|\td\|_{L_q(\dBR^n)}\\
&\le |\lambda|\|d\|_{L_q(\dBR^n)} + |\lambda|\|v\|_{L_q(\dBR^n)}\\
&\le C_{n, q, \eps,\gamma_0}\left\{\|(f, |\lambda|^{1/2} f_d, \nabla f_d, |\lambda|^{1/2} g, \nabla g, |\lambda|h, \nabla^2 h, |\lambda||\nabla'|^{-1}\pd_n h_n) \|_{L_q(\dBR^n)} \right.\\
&\left. \qquad \qquad \qquad + |\lambda| \|g\|_{\hW^{-1}_q(\dBR^n)} + \|d\|_{W^2_q(\dBR^n)} + |\lambda|\|d\|_{L_q(\dBR^n)}\right\}. 
\end{align*}
Theorem \ref{maxregthm2} is also same as above. 

\appendix
\def\thesection{Appendix}
\section{Proof of the estimate for normal component.}\label{appA}
\begin{proof}[Proof of the estimate $\||\lambda||\nabla'|^{-1}\pd_n \psi_n\|_{L_q(\dBR^n)} \le C \|f\|_{L_q(\dBR^n)}$.]
We see that 
\begin{align*}
&|\lambda||\nabla'|^{-1} \pd_n \psi_{\pm n} \\
=& \sum_{k=1}^{n-1}\CF_{\xi}^{-1} \left( |\lambda| \frac{i\xi_n}{|\xi'|}\frac{1}{\rho_\pm \lambda+\mu_\pm |\xi|^2}(\frac{-\xi_n\xi_k}{|\xi|^2})\right) \CF_x f_k + \CF_{\xi}^{-1}\left( |\lambda| \frac{i\xi_n}{|\xi'|}\frac{1}{\rho_\pm \lambda+\mu_\pm |\xi|^2}( 1- \frac{\xi_n^2}{|\xi|^2})\right)\CF_x f_n. 
\end{align*}
All symbols 
\begin{align*}
|\lambda|\frac{i\xi_n}{|\xi'|}\frac{1}{\rho_\pm \lambda+\mu_\pm |\xi|^2}\frac{-\xi_n\xi_k}{|\xi|^2}, \quad |\lambda|\frac{i\xi_n}{|\xi'|}\frac{1}{\rho_\pm \lambda+\mu_\pm |\xi|^2}(1-\frac{\xi_n^2}{|\xi|^2})=|\lambda|\frac{i\xi_n}{|\xi'|}\frac{1}{\rho_\pm \lambda+\mu_\pm |\xi|^2}\frac{|\xi'|^2}{|\xi|^2}
\end{align*}
are bounded and holomorphic in $\lambda\in \Sigma_\eps$, $\xi\in \tilde\Sigma_\eta^n$ for small $\eps, \eta$, where we regard $|\xi'|=\sqrt{\sum_{j=1}^{n-1} \xi_j^2}=A$ and $|\xi|^2=A^2+\xi_n^2$ as complex functions. 
Therefore, by theorem \ref{thm}, we have 
\begin{align*}
\||\lambda||\nabla'|^{-1} \pd_n \psi_n\|_{L_q(\dBR^n)}\le \sum_{\pm} \||\lambda||\nabla'|^{-1} \pd_n \psi_{\pm n}\|_{L_q(\BR^n)}\le C\|f\|_{L_q(\dBR^n)}. 
\end{align*}
The other estimate follows similarly. 
\end{proof}

\subsection*{Acknowledgements} The research was supported by JSPS KAKENHI Grant No.\,19K23408.

\end{document}